\documentclass[11pt, a4paper]{article}

\title{On asymptotic stability  on a center hypersurface at the soliton for even solutions of the NLKG  when $2\ge p> \frac{5}{3}$ }

\author{ Scipio Cuccagna  \and  Masaya Maeda  \and   Federico Murgante  \and  Stefano Scrobogna }

\usepackage[a4paper]{geometry}
\usepackage{empheq}
\usepackage{times}
\usepackage{stmaryrd}
\usepackage[T1]{fontenc}
\usepackage{amssymb, amsmath,  bm, mathrsfs}
\usepackage{amsfonts}
\usepackage[english]{babel}
\usepackage{amssymb,amsthm}
\usepackage{mathtools}
\usepackage{accents}
\usepackage{mathalfa}
\DeclareMathAlphabet{\mathcal}{OMS}{cmsy}{m}{n}
\DeclareFontFamily{U}{mathc}{}
\DeclareFontShape{U}{mathc}{m}{it}%
{<->x*[1.03] mathc10}{}
\DeclareMathAlphabet{\mathscr}{U}{mathc}{m}{it}

\DeclareMathAlphabet{\mathpzc}{OT1}{pzc}{m}{it}

\usepackage{hyperref}
\usepackage[capitalise]{cleveref}
\usepackage{cite}
\allowdisplaybreaks[1]
\usepackage{xfrac}
\usepackage{soul}
\usepackage[utf8]{inputenc}
\usepackage[T1]{fontenc}
\usepackage{enumerate}
\usepackage{accents}
\usepackage{bbm}
\usepackage{thmtools}
\usepackage[author={Max Schlepzig}]{pdfcomment}

\usepackage{todonotes}

\usepackage[makeroom]{cancel}

\usepackage{tikz}
\usetikzlibrary{arrows, automata, backgrounds, shadows, patterns, calc, hobby, plotmarks, shapes}
\usetikzlibrary{shapes.misc}

\tikzset{cross/.style={cross out, draw=black, minimum size=2*(#1-\pgflinewidth), inner sep=0pt, outer sep=0pt},
cross/.default={1pt}}
\allowdisplaybreaks

\makeatletter
\@ifpackageloaded{stix}{%
}{%
  \DeclareFontEncoding{LS2}{}{\noaccents@}
  \DeclareFontSubstitution{LS2}{stix}{m}{n}
  \DeclareSymbolFont{stix@largesymbols}{LS2}{stixex}{m}{n}
  \SetSymbolFont{stix@largesymbols}{bold}{LS2}{stixex}{b}{n}
  \DeclareMathDelimiter{\lBrace}{\mathopen} {stix@largesymbols}{"E8}%
                                            {stix@largesymbols}{"0E}
  \DeclareMathDelimiter{\rBrace}{\mathclose}{stix@largesymbols}{"E9}%
                                            {stix@largesymbols}{"0F}
}
\makeatother

\DeclareSymbolFontAlphabet{\amsmathbb}{AMSb}%

\usepackage{listings}
\usepackage{color}

\definecolor{dkgreen}{rgb}{0,0.6,0}
\definecolor{gray}{rgb}{0.5,0.5,0.5}
\definecolor{mauve}{rgb}{0.58,0,0.82}

\makeatletter
\def\maketag@@@#1{\hbox{\m@th\normalfont\normalsize#1}}
\makeatother

\newcommand{\R}{{\mathbb R}}

\def\im{{\rm i}}

\newcommand{\C}{\mathbb{C}}

\def\({\left(}
\def\){\right)}
\def\<{\left\langle}
\def\>{\right\rangle}
\newcommand{\sech}{{\mathrm{sech}}}
\newcommand{\even}{{\mathrm{even}}}
\newcommand{\odd}{{\mathrm{odd}}}
\newcommand{\supp}{{\mathrm{supp}\ }}

\newcommand{\dd}{\textnormal{d}}
\newcommand{\pare}[1]{\left( #1 \right)}
\newcommand{\norm}[1]{\left\| #1 \right\|}
\newcommand{\av}[1]{\left| #1 \right|}
\newcommand{\bra}[1]{\left[ #1 \right]}

\newcommand{\comm}[2]{\bra{ #1 \  , \   #2  }}
\newcommand{\xfrac}[2]{\left. #1 \middle/ #2 \right. }
\newcommand{\set}[1]{\left\{ #1 \right\}}

\newcommand{\defeq}{\vcentcolon=}
\newcommand{\eqdef}{=\vcentcolon}

\newcommand{\spp}{\sigma_{\textnormal{pp}}}
\newcommand{\Span}[1]{\textnormal{Span}\set{#1}}

\renewcommand{\Re}{\textnormal{Re}}
\renewcommand{\Im}{\textnormal{Im}}
\newcommand{\RN}[1]{%
  \textup{\uppercase\expandafter{\romannumeral#1}}%
}

\newcommand{\bmz}{{\bm z}}

\newcommand{\bmPhi}{{\bm \Phi}}

\newcommand{\cT}{\mathcal{T}}

\newcommand{\bR}{\mathbb{R}}

\newcommand{\cH}{\mathcal{H}}
\newcommand{\cO}{\mathcal{O}}

\newcommand{\bC}{\mathbb{C}}

\newcommand{\ii}{\mathrm{i}}

\newcommand{\ps}[2]{\pare{  #1 \ \middle| \ #2 }}
\newcommand{\psc}[2]{\left\langle  #1 \ ,  \ #2 \right\rangle}

\theoremstyle{theorem}
\newtheorem{theorem}{Theorem}[section]

\newtheorem*{theorem*}{Theorem}

\newtheorem{proposition}[theorem]{Proposition}
\newtheorem{lemma}[theorem]{Lemma}

\theoremstyle{definition}

\newtheorem{rem}[theorem]{Remark}
\newtheorem{remark}[theorem]{Remark}
\newtheorem{assumption}[theorem]{Assumption}
\newtheorem{notation}[theorem]{Notation}

\newtheorem{step}{Step}
\makeatletter
\@addtoreset{step}{subsection}
\makeatother

\makeatletter
\@addtoreset{substep}{step}
\makeatother

\makeatletter
\@addtoreset{proofpart}{theorem}
\makeatother

\makeatletter
\@addtoreset{proofsubpart}{part}
\makeatother
\newtheorem{case}{Case}
\makeatletter
\@addtoreset{case}{theorem}
\makeatother

\numberwithin{equation}{section}

\begin{document}

\maketitle

\begin{abstract} We extend the result of   Kowalczyk, Martel and Munoz \cite{KMM2022} on the existence, in the context of spatially even solutions,  of asymptotic stability on a center hypersurface at the soliton of the defocusing power  Nonlinear Klein Gordon Equation with $p>3$,  to the case $2\ge p> \frac{5}{3}$. The result is attained performing new and refined estimates that allow to close the argument for power law in the range $2\ge p> \frac{5}{3}$.
\end{abstract}

\begin{footnotesize}
\tableofcontents
\end{footnotesize}

\section{Introduction}

We consider the Nonlinear-Klein Gordon Equation (NLKG) on the line
\begin{equation}\label{eq:nlkg1}
  \partial _t^2 u_1- \partial _x^2 u_1+u_1 - f(u_1) =0, \quad { (t,x)\in \R\times \R}
\end{equation}
  where $  f(u_1)\defeq
  |u_1|^{p-1}u_1 $ with $ p_3\defeq2\ge p>p_4\defeq\frac{5}{3} $.
As in      Krieger et al.  \cite{KNS2012},
 Kowalczyk et al. \cite{KMM2022} and    Li  and   L\" uhrmann \cite {LL2023}, we  consider only even solutions, eliminating translations  and simplifying the problem. Translations pose  difficulties,
 see  the complications in the virial inequalities  in  Kowalczyk et al. \cite{KMMV2021}.

\noindent Setting ${\bm u}\defeq(u_1,u_2) ^ \intercal  \defeq(u_1,\dot u_1)^ \intercal  $, where in the sequel $\dot u=\partial _t u$ and $u'=\partial _x u$, we obtain the system
\begin{align}\label{NLKG}
\dot{{\bm u}}=\mathbf{J} \begin{pmatrix}
- \partial _x^2   +1 & 0 \\ 0 & 1
\end{pmatrix} {\bm u}+  {f}(u_1) \mathbf{j}, \text{ with } \mathbf{J}\defeq\begin{pmatrix}
0 & 1 \\ -1 & 0
\end{pmatrix}
\end{align}
and with $\mathbf{j}=(0,1)^\intercal$, later also $\mathbf{i}=(0,1)^\intercal$ (not to be confused with the imaginary unit $\im$).
We denote by $Q$ the standing wave, given, see for example in \cite[formula (3.1)]{CGN2007}, by
\begin{align}\label{eq:solit1}
 Q(x)\defeq \left( \frac{p+1}{2}  \right) ^{\frac{1}{p-1}}  \sech    ^{\frac{2}{p-1}}\left(  \frac{p-1}{2}x   \right) .
\end{align}
The system \eqref{NLKG} is a  Hamiltonian system for the symplectic form
\begin{align} &
\Omega({\bm u},{\bm v})\defeq\<\mathbf{J}^{-1}{\bm u},{\bm v}\>,      \label{eq:inner1} \\
&  \label{eq:inner0}
\<{\bm u},{\bm v}\> \defeq  \pare{ {\bm u},\overline{{\bm v}}} ,
&&
  \pare{
 {\bm u},  {\bm v} }  \defeq\int _{\R} {{\bm u}}(x) ^ \intercal \ {\bm v}(x) dx,
\end{align}
and  Hamiltonian,  or  energy function, given by
\begin{align}\label{eq:energy}
E({\bm u})=\frac{1}{2} \left( \norm{ u'_1} ^2 _{L^2\left( \R \right) }+ \norm{ u _2 } ^2_{L^2\left( \R \right) }  \right)-\int_{\R}\dfrac{\av{u_1}^{p+1}}{p+1}\,dx .
\end{align}
We consider the  space of even functions
\begin{equation*}
\bm{\mathcal{H}  }   ^{1}_{\even}
=
\set{{\bf u}=\pare{u_1, u_2} ^ \intercal  \ \middle| \ {\bf u} \in \boldsymbol{\mathcal{H}  }^1 \ \text{and} \ {\bf u}\pare{x} = {\bf u}\pare{-x} }
\end{equation*}
  with norm \begin{align}\label{eq:ennrom}
 \|{\bm u}\|_{\boldsymbol{\mathcal{H}  }   ^{1}   }^2=\|u_1\|_{H^1 }^2+\|u_2\|_{L^2 }^2.
\end{align}
We recall that Kowalczyk et al. \cite{KMM2022} proved, for $p>3$, that there exists an orbitally stable central hypersurface   in $\boldsymbol{\mathcal{H}  }   ^{1}_{\even}$. With minor modifications, their proof   extends also to our case  $2\ge p > \frac{5}{3}$, so we assume it.  Furthermore, Kowalczyk et al. \cite{KMM2022} proved that
$Q$ is asymptotically stable for solutions on this hypersurface.   Li  and   L\" uhrmann \cite {LL2023} considered    $p=p_3=2$,  with $f(u_1)=u_1^2$.
 In  Li  and   L\" uhrmann \cite {LL2023}  the  threshold resonance $\mu _3=1$  is potentially an obstruction to the proof, but since the corresponding generalized eigenfunctions  are odd in $x$, they are naturally orthogonal to the elements of $\boldsymbol{\mathcal{H}  }   ^{1}_{\even}$.  So the threshold resonance $\mu _3=1$
does not affect the argument, in line with the discussion in Kowalczyk and Martel \cite{KM22} on the asymptotic stability of odd perturbations of the standing kink of the $\phi ^4$ model, where the threshold resonance is \textit{even}, and so does not affect odd solutions.  In the case $p=2$, for the non smooth nonlinearity $f(u_1)=|u_1| u_1 $ treated here,  the same will happen in our proof.

  Notice that the case   $2<p<3$  is simpler.  Case $p=3$ is different and still open, because of an  \textit{even} threshold resonance,   certainly affecting even solutions of \eqref{eq:nlkg1}. For a partial finite time result, framed with   very different techniques, see  L\" uhrmann and Schlag \cite{LS2023}. Very recently  Palacios and Pusateri \cite{PalPus24}  obtained a finite time result mixing arguments similar to  Kowalczyk et al. \cite{KMM2022}  and phase space analysis.
 Finally, as $p\rightarrow 1 ^+$, besides the case of the $p>1$ for which there is a resonance, which form a discrete subset in $(1,\infty)$, there is also an increase of the number of eigenvalues $p\rightarrow 1 ^+$, with some converging to 0, as can be seen in the pictures in Chang et al. \cite{CGN2007}. While   methods of proving dispersion by means of dispersive estimates  or Strichartz estimates, with the Duhamel formula,  are not particularly robust and are not applicable for $p$ sufficiently close to 1, in fact for $p<3$, low values of $p$ arbitrarily close to 1 could in principle be treated by the dispersion theory in Kowalczyk et al. \cite{KMM2022}. Furthermore,  the appearance  of a complicated set of eigenvalues, is not, by itself, an impossible problem. The corresponding discrete modes  slow the rate of stabilization, as discovered computationally  in \cite{BCS2011} and observed in a rigorous context   in \cite{BP1995,SW1999,TY3,zhousigal}  as well as in many other papers, most of them quoted in \cite{CMsurv}, to which we refer for an extensive bibliography. But this slowing of the stabilization does not affect proofs of stabilizations similar to  \cite{Cu2011}, which is older, or to \cite{CMS2023}.
  In the stable case,  there is a well known stabilization mechanism that all the references utilize,  denominated Nonlinear Fermi Golden Rule (FGR), was initially thought by Sigal \cite{Sigal1993}, involving the   nonlinear interaction of the continuous modes with the continuous ones. Using the FGR,   older references like \cite{Cu2011}, which used normal forms arguments, and  more recent ones like \cite{CMS2023}, which used the so called Refined Profile, were able to deal with generic configurations of eigenvalues. Finding the Refined Profile is a relatively simple operation
 involving Taylor expansions  of sufficiently high order, depending  on how close to the ground state get the excited states, of the function $f(\boldsymbol{\Phi}\pare{\bm z}_1(x))$, which appears later in
 \eqref{eq:rp1}. In the context of the present paper, the low regularity  of the power function however becomes a problem.   Notice though that  $f(\boldsymbol{\Phi}\pare{\bm z}_1(x))$ is smooth at values of $({\bm z},x)$ for which $\boldsymbol{\Phi}\pare{\bm z}_1(x)\neq 0$.
The gist of the argument and the intuition behind this paper is that  by an appropriate partitioning of spacetime, for example with the distinction in \cref{case:one} and \cref{case:two} in the proof of \Cref{lem:rpcorr}, we can take up to 2 derivatives even when $p<2$. Furthermore, to get estimates like \eqref{eq:Ijest2}, and specifically the fact that $ Q ^{p-2}\varphi ^2_2\sim e^{-|x|} $ (see later for the notation), we neutralize the growth of  $ Q ^{p-2}$ as $x\to  \infty$    with the decay of $\varphi ^2_2$, see later in \eqref {eq:asympt_eigenf}.   It is possible, in principle, that a more systematic analysis of this type  will allow to treat cases with $p$   closer to 1. The specific idea utilized here breaks down
for $p_5= \frac{3}{2}<p<p_4= \frac{5}{3}$, which would have  an additional term
$ Q ^{p-2}\varphi ^2_4 \sim  e^{ (2p-3)|x|}  $,  unbounded as $x\to  \infty$.

 \noindent There is a large literature, beyond the above references, like the influential  Kowalczyk et al. \cite{KMM2017}, which gave great impulse to the literature on 1 dimensional stabilization problems.
  There are many    papers with a different framework and origin,  involving dispersive estimates, like \cite{CLL2020,CP2022,DM2020,GP2022,GPZ2023, LP2022} and   references therein. A quite extensive list of references known to the  authors knowledge  which had appeared before the year 2019  can be found in    \cite{CMsurv}.

\subsection{On the linearization of \Cref{NLKG} around $ Q $} Here we give some information about the linearization of   \Cref{NLKG} around $ Q $, cfr. \cite{CGN2007}.
Key to the analysis is the operator (often denoted, in the literature, by $L_+$),
\begin{align}\label{eq:lin1}
  L_0\defeq - \partial _x^2   +1-p Q ^{p-1} .
\end{align}
  It is well known, see Grillakis et  al. \cite[\S 6]{GSS1},  that $Q$ is orbitally unstable.  Indeed, if we set
\begin{align}\label{eq:lin1.1}
\mathbf{L}_0\defeq\begin{pmatrix}
L_0 & 0 \\ 0 & 1
\end{pmatrix},
\end{align}
the linearization
    of \eqref{NLKG}   at $Q \mathbf{i}$ is
  \begin{equation*}
  \mathbf{J}\mathbf{L}_0
  \defeq
  \pare{
\begin{array}{cc}
0&1 \\ -L_0 & 0
\end{array}
  } ,
  \end{equation*}
   admits   point spectrum
\begin{equation*}
\spp \pare{\mathbf{J}\mathbf{L}_0 } = \set{ \nu \defeq \pm \sqrt{-\mu} \ \middle| \ \mu \in \spp \pare{L_0}
}.
\end{equation*}
In particular the elements of $ \spp \pare{L_0} $   as functions of $ p $ are
\begin{align}
\label{eq:pM_kj}
p_N \defeq  \begin{cases}\frac{N+1}{N-1} \text{  for } N > 1,\\  \infty \text{ if } N=1, \end{cases}
&&
 k_j \defeq \frac{p+1}{2}-\frac{j(p-1)}{2} , \quad j\in\mathbb{N},
\end{align}
if $ p_{N+1}  < p<p_N$   then  $ \spp\pare{L_0} $     has  $ N+1 $ elements    and they   are explicitly  given by
\begin{align} \label{eq:eigenv} \mu _m \defeq 1-k^2_m,
&& m=0, \ldots , N.
   \end{align}
%Let us thus remark, from \cref{eq:eigenv}, that when $ p>1 $ we have that if $ m=0 $ then $ \mu_0 < 0 $. This implies that $ {\bf J}\mathbf{L}_0 $, has a positive, and thus spectrally and also dynamically unstable eigenvalue $\nu _0=\sqrt{-\mu _0}>0$. Moreover it is known (see \cite{DT1979}) that the operator $ L_0  $ has  continuous spectrum $ \sigma_{\text{c}}\pare{L_0} = [1, \infty) $ and generalized eigenfunctions $ \pare{ \mathsf{e}\pare{\xi} }_{\xi \neq 0} $ \todoSS{Please Scipio, check}\begin{equation}\label{eq:geneigenfLQ}L_0 \ \mathsf{e}\pare{\xi} = \pare{1+\xi^2} \mathsf{e}\pare{\xi}.\end{equation}Let us notice that defining\begin{equation}\label{eq:geneigenfJLQ}\bm{\mathsf{e}}_{\pm}\pare{\xi} \defeq \binom{1}{\pm \ii \sqrt{1+\xi^2}} \mathsf{e}\pare{\xi},\end{equation}we have\begin{equation*}\mathbf{JL}_0 \  \bm{\mathsf{e}}_{\pm}\pare{\xi} = \pm \ii \sqrt{1+\xi^2}\  \bm{\mathsf{e}}\pare{\xi}.\end{equation*}

\noindent When $ p=p_N $ for some $ N\geq 1 $, then from \cref{eq:eigenv} we have $ \mu_N =1 $ and the linear operator $ {\bf J}\mathbf{L}_0 $ has a threshold  resonance at the boundary of the continuous spectrum. Here we will consider \begin{align}\label{def:valueN}
  N =\begin{cases}
 3   \text{   if }  p_3=2> p>p_4=\frac{5}{3}  \\
 2 \text{   if }  p=2
\end{cases}
\end{align}
  and    we  will use
 \begin{align}\label{eq:valk}
   k_0 = \frac{p+1}{2}, &&  k_1=1  , && k_2= \frac{3-p}{2} , && k_3= 2-p,  &&  k_4= \frac{5-3p}{2}.
 \end{align}
   Like in   Krieger et al.  \cite{KNS2012},
 Kowalczyk et al. \cite{KMM2022} and    Li  and   L\" uhrmann \cite {LL2023}, we will prove the   asymptotically stability of $Q$ in the central hypersurface of Kowalczyk et al. \cite{KMM2022}. The constraint  $2\ge p> \frac{5}{3}$  renders the nonlinearity more problematic  compared  \cite{KNS2012,KMM2022,LL2023},  in terms of lower   differentiability and bigger strength. %We will denote with $ D_0 $ and $ \mathbf{D}_0 $ the unperturbed operators\begin{align*}D_0 \defeq -\partial_x^2 + 1,&&{\bf D}_0 \defeq\begin{pmatrix}- \partial _x^2   +1 & 0 \\ 0 & 1\end{pmatrix}  . \end{align*}
For % \footnote{The operators $ S_j $ here correspond to the operators $ A_j $ in \cite[Proposition 1.11]{CMS2023}}
\begin{align}\label{eq:defSj}
   S_j \defeq\partial_{x}-k_j \tanh \left(  \frac{p-1}{2}x   \right)
 \end{align}
and for the eigenvalues $ \mu_j $   in \cref{eq:eigenv} ,
 we have $
 \ker \pare{L_0-\mu_j} = \Span{\varphi_j} ,$
 with
 \begin{align}
 \label{eq:eigC}
% \begin{aligned}
 \varphi_0 \defeq Q ^{k_0} ,  &&
   \varphi_1 \defeq S^\ast_0 Q^{k_1} =   \frac{p-1}{2} \ Q'  ,  && \varphi_2 \defeq  S_0^*S_1^* Q ^{k_2} , && \varphi_3 \defeq S_0^*S_1^* S_2^* Q ^{k_3}.
% \end{aligned}
 \end{align}
We emphasize that the
$ S^\ast_j $ change the symmetry of a function,  so that we have the very general fact that $ \varphi_0 $ and $ \varphi_2 $ are even while $ \varphi_1 $ and $ \varphi_3 $ are odd.  It is elementary and well known \cite{DT1979}, that
 \begin{equation}
 \label{eq:asympt_eigenf}
 \av{\varphi_j\pare{x}}\sim e^{-k_j\av{x}} \text{ as }   \av{x}\to \infty  .
 \end{equation}
Using   \eqref{eq:eigC}, for
\begin{align}\label{eq:matr_eig1}
{\mathbf{Y}}_{\pm }\defeq\begin{pmatrix}
\varphi_0\\ \pm  \nu _0\varphi_0
\end{pmatrix} ,
&&
{\mathbf{Z}}_{\pm }\defeq\begin{pmatrix}
\varphi_0\\ \pm  \nu _0 ^{-1}\varphi_0
\end{pmatrix} ,
\end{align}
we have
 \begin{align*}
\mathbf{J}\mathbf{L}_0  {\mathbf{Y}}_{\pm }   = \pm \nu _0   {\mathbf{Y}}_{\pm }
 .
\end{align*}
We set
\begin{align}\label{eq:matr_eig1}
\boldsymbol{\Phi}_0\defeq \frac{  {\mathbf{Y}}_{+ }- \im {\mathbf{Y}}_{- }}{2} = \frac{1}{2}\binom{1-\ii}{\pare{1+\ii}\nu_0} \varphi_0 \  .
\end{align}
Then,
\begin{align}\label{eq:matr_eig12}
\mathbf{J}\mathbf{L}_0 \boldsymbol{\Phi}_0 =  \nu _0 \bar{\boldsymbol{\Phi}}_0.
\end{align}
Similarly, setting $ \lambda  = \sqrt{\mu _2}>0 $ we define
\begin{align}\label{eq:matr_eig2}
\boldsymbol{\Phi}_2\defeq\begin{pmatrix}
 1 \\ \im \lambda
\end{pmatrix} \ \varphi_2 \  ,
\end{align}
so that
\begin{align}\label{eq:matr_eig3}
\mathbf{J}\mathbf{L}_0 \boldsymbol{\Phi}_2=\im \lambda  \boldsymbol{\Phi}_2   &&
 \mathbf{J}\mathbf{L}_0 \bar{\boldsymbol{\Phi}}_2=-\im \lambda \bar{\boldsymbol{\Phi}}_2.
\end{align}
%Let us notice that since $ \varphi_1 $ and $ \varphi_3 $ are odd they shall play no role in the forthcoming analysis of the problem. \\
Notice that for \begin{align}\label{def:calA}
 \mathcal{A} =S_0\cdots S_N   , \text{ for $N$  as in \eqref{def:valueN},}
%&& N =\begin{cases} 3   \text{   if }  p_3=2> p>p_4=\frac{5}{3}  \\ 2 \text{   if }  p=2\end{cases} ,
\end{align}
 we have
 \begin{align}\label{eq:DarConj2}
\mathcal{A}^*  L_0=L_{N+1}\mathcal{A} ^*,
\end{align}
where
\begin{align}\label{eq:opLj}
 L_{j}\defeq  - \partial _x^2   +1-k_{j-1}k_{j }\frac{2}{p-1}Q ^{p-1}.
\end{align}
The procedure outlined in \Cref{eq:opLj,eq:DarConj2,def:calA} transforms the linearization of \cref{NLKG} around $ Q $ from
$
\dot U = {\bf J} \mathbf{L}_0 U ,
$
into the new linear equation, in the auxiliary variable $ V : = \mathcal{A} U $,
\begin{align}
\label{eq:linearNLKG_2}
\dot V = {\bf J } {\bf L}_{ N+1} V ,
&&
{\bf L}_{ N+1}
: =
\left(
\begin{array}{cc}
L_{N+1} & 0
\\
0 & 1
\end{array}
\right) ,
\end{align}
%the advantage of \cref{eq:linearNLKG_2} compared to \cref{eq:linearNLKG} is that
where for  $ p_3 >  p>p_4 $ the operator $L_{N +1}=L_4$ has a repulsive potential, in the sense that
\begin{align}\label{eq:repuls}
  -k_{3}k_{4 }\frac{2}{p-1}   x \left(Q ^{p-1} \right) ' <0 \text{  for all } x\neq 0,
\end{align}
while for $p=2$ we have $L_{N +1}=L_3=- \partial _x^2   +1$  but fortunately the parity of the solutions obviates to the lack of this repulsivity. \\

 \noindent Given a constant $\kappa >0$, we consider the spaces defined by the following norms,
\begin{align}\label{eq:enwei}
 \| {\bm u}\|_{ \boldsymbol{\mathcal{H}}^{1 }_{-\kappa }}\defeq \| \sech \left( \kappa x \right)   {\bm u}\|_{ \boldsymbol{\mathcal{H}}^{1 }  }  , &&
  \| {\bm u}\|_{ L^{2 }_{-\kappa }}\defeq \| \sech \left( \kappa x \right)   {\bm u}\|_{ L^2  }.
\end{align}
In the sequel   $\kappa>0$ is fixed and smaller than a certain finite set of positive numbers coming up in the proof. For ${\bm z}=(z_1,z_2) ^\intercal \in \C^2$, we set
\begin{align}\label{eq:rp}&
 \boldsymbol{\Phi}\pare{\bm z} \defeq Q \mathbf{i} +  \tilde{\boldsymbol{\Phi}}[{\bm z}] ,
  \text{ where }
 \tilde{\boldsymbol{\Phi}}[{\bm z}] \defeq z_1 \boldsymbol{\Phi}_0+z_2 \boldsymbol{\Phi}_2+  \overline{z}_1 \bar{\boldsymbol{\Phi}}_0+\overline{z}_2 \bar{\boldsymbol{\Phi}}_2
   = 2 \Re\pare{ z_1 \boldsymbol{\Phi}_0+z_2 \boldsymbol{\Phi}_2}  \\& = z_{1R}  {\mathbf{Y}}_{+ }-z_{1I}  {\mathbf{Y}}_{- } + 2\varphi _2  \begin{pmatrix}
z_{2R}\\  -\lambda z_{2I}
\end{pmatrix} \text{ where }z_{jR}=\Re z_j \text{ and }z_{jI}=\Im z_j   \nonumber
  .
\end{align}
By the linearity of $\tilde{\boldsymbol{\Phi}}[{\bm z}]$   in ${\bm z}$, we obviously have
\begin{equation}
\label{eq:DzPhi}
D_{\bm z}\boldsymbol{\Phi} ({\bm z}){\bm w} = \tilde{\boldsymbol{\Phi}}\bra{\bm w},
\end{equation}
where $D_{\bm z}\varphi({\bm z}){\bm w}:=\left.\frac{d}{ds}\right|_{s=0}\varphi({\bm z}+s{\bm w})$.
We set
\begin{align}\label{eq:cont_mod}
 \boldsymbol{\mathcal{H}}^{1 }_{c} \defeq \set{ \boldsymbol{\eta} \in \boldsymbol{\mathcal{H}  }   ^{1}_{\even}(\R , \C ^2 ) \ \middle| \  \left \langle  \mathbf{J}\boldsymbol{\eta} , D_{\bm z} \boldsymbol{\Phi}\pare{\bm z} \boldsymbol{\Theta} \right  \rangle  =0 \text{ for all }\boldsymbol{\Theta}\in \C^2} .
\end{align}
There exists a nonzero  $ {\bf g}_2 \in L^\infty\pare{\bR; \bC^2} $  such that $ {\bf JL}_0{\bf g}_2= \ii \ 2\lambda {\bf g}_2 $. It is of the form $ {\bf g}_2 = \binom{1}{ 2\ii \lambda} g $ with $ L_0 g = 4\lambda^2 g $, i.e. $ g $ is a \emph{ distorted plane-wave} for the Schr\"odinger operator $ L_0 $, see \cite{DT1979}.
\begin{assumption}[Fermi Golden Rule]\label{ass:FGRintro}
 We will consider the    $ p\in \left(\xfrac{5}{3}, 2 \ \right] $  for which
\begin{equation*}
- p \pare{p-1} \psc{Q^{p-2}\varphi_2^2}{g}  \neq 0,
\end{equation*}
where $ Q $ and $ \varphi_2 $ are defined in \cref{eq:solit1,eq:eigC} respectively.
\end{assumption}

\begin{rem}
    \Cref{ass:FGRintro} is proved   for $p=2$ by   Li  and   L\" uhrmann \cite [Lemma 3.1] {LL2023} and  for   $p $ close to 2 remains true by continuous dependence on $p$. Since this dependence is also analytic,    \Cref{ass:FGRintro}  is true for all but  a discrete subset of $p\in   [  5/3, 2  ) $. It would be interesting to check computationally if Assumption  \ref{ass:FGRintro} holds for  all $  p\in \left(\xfrac{5}{3}, 2 \ \right] $. This is feasible especially because $Q$, $\varphi_2$  and also $g$ are known explicitly, but we did not try it.
\end{rem}

Our main result is the following.

\begin{theorem}\label{thm:main1}
Under \Cref{ass:FGRintro},
for any $a>0$ and $ \epsilon >0$
  there exists $\delta_0>0$ such that for  $\delta \in (0,\delta_0)$ if
  \begin{align}\label{eq:cond_sta}
    \sup _{t\ge 0} \norm{ {\bm u} (t) -  Q \mathbf{i} }  _{\boldsymbol{\mathcal{H}  }   ^{1}_{\even}(\R , \R ^2 )}<\delta
  \end{align}
    then there exists ${\bm z}\in C ^1\pare{ \R , \C ^{2} } $ and  $ \boldsymbol{\eta }\in C^0 \pare{ \R ,  \boldsymbol{\mathcal{H}}^{1 }_{c} }$ such that we have a global representation
 \begin{align}
 {\bm u}(t) =     {\boldsymbol{\Phi}}\pare{ {\bm z}(t) }   + \boldsymbol{\eta }
   (t)
 \label{eq:main1}
\end{align}
such that
\begin{align}&
 \int _{I }\(   \|     \boldsymbol{\eta }  (t) \| ^2 _{\boldsymbol{\mathcal{H}}^{1} _{-a}(\R )}
 + |z|^4 \) \dd t   \le \epsilon,     \label{eq:main20}
 \\&
 \lim_{t\to+\infty} \norm{     \boldsymbol{\eta }  (t) } ^2 _{\boldsymbol{\mathcal{H}}^{1} _{-a}(\R )}=0 \text{ and }  \label{eq:main2}
\\&
\lim_{t\to  \infty} {\bm z} (t) =0     \text{  .  }   \label{eq:main3}
\end{align}
\end{theorem}

For  $\delta _0>0$ we  set
\begin{align}& \mathcal{B}_0\defeq \set{    \boldsymbol{\varepsilon } \in  \boldsymbol{\mathcal{H}  }   ^{1}_{\even}(\R , \R ^2 )  \ \middle| \   \|     \boldsymbol{\varepsilon}    \| ^2 _{\boldsymbol{\mathcal{H}}^{1}  }  <\delta _0   \text{  and }   \<   \boldsymbol{\varepsilon } , \mathbf{Z} _+\> =0} .  \label{eq:stbm}
\end{align}
The following   existence and uniqueness of a stable central hypersurface, can be proved like   Theorem 2 in  Kowalczyk et al. \cite{KMM2022}.
\begin{theorem}\label{thm:main2}
There exist  $C,\delta _0>0$ and a Lipschitz function $ h: \mathcal{B}_0 \to \R$
 with
 \begin{align}\label{eq:Lipfun}
     \text{$h(0)=0$ and $|h(\boldsymbol{\varepsilon })|\le  C \| \boldsymbol{\varepsilon } \|  _{\boldsymbol{\mathcal{H}}^{1}  } ^{ \frac{p+1}{2}} $}
 \end{align}
 such that, for
 \begin{align}\label{eq:stabman}
     \mathcal{N}\defeq\set{ Q \mathbf{i} + \boldsymbol{\varepsilon }+ h(\boldsymbol{\varepsilon }) \mathbf{Y}_+  \ \Big| \  \boldsymbol{\varepsilon }\in      \mathcal{B}_0             } ,
  \end{align}
 we have the following:
 \begin{description}
   \item[a]  if ${\bm u}_0\in   \mathcal{N}$, then the corresponding solution of \eqref{NLKG} is  defined  for  all $t\ge 0$   and  \begin{align}
  \sup _{t\ge 0} \| {\bm u} (t) - Q \mathbf{i}  \| _{\boldsymbol{\mathcal{H}  }   ^{1}_{\even}(\R , \R ^2 )}\le C  \| {\bm u}_0 - Q \mathbf{i}  \| _{\boldsymbol{\mathcal{H}  }   ^{1}_{\even}(\R , \R ^2 )} .
 \label{eq:main21}
\end{align}
   \item[b] if a a solution ${\bm u} (t)$ defined  for  all $t\ge 0$  satisfies   \begin{align}
    \| {\bm u} (t) - Q \mathbf{i}  \| _{\boldsymbol{\mathcal{H}  }   ^{1}_{\even}(\R , \R ^2 )}   < \frac{\delta _0}{2}
 \label{eq:main22}
 \end{align} then ${\bm u} (t) \in  \mathcal{N}$ for all $t\ge 0$.

 \end{description}

\end{theorem} \qed

\begin{remark} Notice that by $p\le 2$ the exponent $ \dfrac{p+1}{2}$ in \eqref{eq:Lipfun} is smaller than the exponent $3/2$ in  \cite{KMM2022}. In fact here we have inequalities for example like
  \begin{align*}
      \left |   f(Q+v) - f(Q+\tilde{v})-f'(Q) (v-\tilde{v})\right |\lesssim \( |v| ^{\frac{p-1}{2}}+|\tilde{v}| ^{\frac{p-1}{2}} \) |v-\tilde{v}|
  \end{align*}
while in Proposition 4   \cite{KMM2022} the exponent $\dfrac{p-1}{2}$ is replaced by the larger value 1, Theorem \ref{thm:main2} can be obtained    mimicking   almost exactly the discussion in sections 6.2 and 6.3 \cite{KMM2022},  with  only  some minor modifications in the exponents.

\end{remark}

As a corollary, it follows that  the conclusions of  Theorem  \ref{thm:main1}   are true for all the solutions of \eqref{NLKG} in the stable center hypersurface $\mathcal{N}$.   This paper is  focused uniquely on the proof of Theorem  \ref{thm:main1} and, for dispersion, uses the framework of Kowalczyk et al., even though for technical results   often refers to \cite{ CM2022, CM2023,CMS2023}.

\subsection{Refined profile and Fermi Golden Rule }\label{sec:refprof}

 In our earlier work \cite{CMS2023}, we referred to the notion of \textit{Refined Profile}. Our problem here requires only a special case of this notion.
We introduce the notation
\begin{align}\label{eq:rp--1}&  \tilde{{\bm z}}_0\defeq\tilde{{\bm z}}_0[{\bm z}]\defeq (\nu _0 \overline{z}_1, \im \lambda z_2) .
\end{align}
Then $ \boldsymbol{\Phi}\pare{\bm z} = \pare{ \boldsymbol{\Phi}\pare{\bm z}_1, \boldsymbol{\Phi}\pare{\bm z}_2  }^\intercal $, satisfies, using \cref{eq:rp--1,eq:rp,eq:matr_eig3,eq:matr_eig12}, the equation
\begin{equation}\label{eq:LQonPhitilde}
D_{\bm z}\boldsymbol{\Phi}\pare{\bm z}\tilde{{\bm z}}_0 = {\bf J}\mathbf{L}_0 \tilde {\bm \Phi}\bra{z}.
\end{equation}
Thus using the fact that $ Q {\bf i} $ is a stationary solution of \cref{NLKG} %and the fact that the potential in \cref{eq:lin1} is $ - f'\pare{Q} $
we obtain
\begin{align}\label{eq:rp1}&
D_{\bm z}\boldsymbol{\Phi}\pare{\bm z}\tilde{{\bm z}}_0 = \mathbf{J}
\begin{pmatrix}
	-\partial_x^2+1 & 0\\ 0 &1
\end{pmatrix}
%\mathbf{L}_0
 \boldsymbol{\Phi}\pare{\bm z} + f(\boldsymbol{\Phi}\pare{\bm z}_1) \mathbf{j} + \widehat{\bm R}\pare{\bm z}, \text{ with}
\\ \label{eq:errp1}&
\widehat{\bm R}\pare{\bm z} := -\left(  f(\boldsymbol{\Phi}\pare{\bm z}_1)   - f(Q)-f' (Q)  \tilde{\boldsymbol{\Phi}}[{\bm z}]_1     \right)  \mathbf{j} .
\end{align}
It is useful to improve the remainder $ \widehat{\bm R}\pare{\bm z}$  as follows.

\begin{lemma} \label{lem:rpcorr} There exist constants $C_1,\delta _1 >0$ and  a map $ \tilde{{\bm z}}_R : D _{\C^2} (0, \delta _1) \to \C^2 $ s.t. the following   holds:
\begin{enumerate}[\bf i)]

\item \label{item:rem_better_3} for
\begin{equation}
\label{eq:errp2}  {\bm R}\pare{\bm z} : = D_{\bm z}\bm{\Phi}\pare{\bm z}\ \tilde{{\bm z}}_R \pare{{\bm z}}
+\widehat{\bm R}\pare{\bm z},
\end{equation}
we have
\begin{align} \label{R:orth} &
 \left< \mathbf{J}{\bm R}\pare{\bm z}  , D _{{\bm z}}\boldsymbol{\Phi}\pare{\bm z}   \Theta  \right>  =0 \text{ for all   $\Theta \in \C ^2 $  and all $ {\bm z} \in D _{\C^2} (0, \delta _1)$;}
\end{align}

\item  \label{item:rem_better_20} we have
 \begin{equation}\label{item:rem_better_1}
 	\av{ \tilde{{\bm z}}_R ({\bm z})} \le C _1 |{\bm z}|^2   \text{  for any ${\bm z}\in D _{\C^2} (0, \delta _1)$;}
 \end{equation}
\item \label{item:rem_better_1.1} $\tilde{{\bm z}}_R \in C^1  \pare{ D _{\C^2} (0, \delta _1), \C^2  }$;
\item \label{item:rem_better_2} $\tilde{{\bm z}}_R$ is twice differentiable in 0.

\end{enumerate}

\end{lemma}

\begin{proof}
We prove  \Cref{item:rem_better_3}.
  First of   all,  \eqref{R:orth}  is equivalent to
\begin{align}\label{R:orth1} &
 \<   \mathbf{J} D _{{\bm z}}\boldsymbol{\Phi}\pare{\bm z}  \tilde{{\bm z}}_R  , D _{{\bm z}}\boldsymbol{\Phi}\pare{\bm z}   \Theta  \>  =-\< \mathbf{J}\hat{{\bm R}}\pare{\bm z}   , D _{{\bm z}}\boldsymbol{\Phi}\pare{\bm z}   \Theta  \> , \text{     for all   $\Theta \in \C ^2 $  and all $ {\bm z} \in D _{\C^2} (0, \delta _1)$.}
\end{align}
By the definition  of $\boldsymbol{\Phi}\pare{\bm z}$,  varying $\Theta$ in a $\R$--basis of $\C ^2$,
\cref{R:orth1} is a linear system with unknown $ \tilde{{\bm z}}_R$ with constant coefficients. The linear operator acting on $ \tilde{{\bm z}}_R$ is invertible, because the affine space spanned by $\boldsymbol{\Phi}\pare{\bm z}$ is symplectic.  In particular,  varying $\Theta$ in the standard $\R$--basis of $\C ^2$ and utilizing the notation in \cref{eq:rp}, the matrix is, here $z _{a,R}=\Re z _{a}$  and  $z _{a,I}=\Im z _{a}$, for $a=0,2$, \small
\begin{align}\label{symplmatr}  &
\left \{   \<   \mathbf{J} \partial  _{z _{a,A}}\boldsymbol{\Phi}\pare{\bm z}    , \partial  _{z _{b,B}}\boldsymbol{\Phi}\pare{\bm z}      \>   \right \} _{ a,b =0,2 \text{ and } A,B =R,I }   = \begin{pmatrix}
	\< J  \mathbf{Y}_-,\mathbf{Y}_+\> J
		&  0\\
	0
		&  \lambda \| \varphi _2 \| _{L^2 (\R )}^2 J
\end{pmatrix} ,
\end{align}\normalsize
which is obviously invertible since $\< J  \mathbf{Y}_-,\mathbf{Y}_+\> = -2\nu _0 \| \varphi _0\| ^2 _{L^2}$.  This   matrix reappears   later in \Cref{lem:modbound}.

We prove \Cref {item:rem_better_20}.
 From the above  we have
\begin{align}
		\av{ \tilde{{\bm z}}_R ({\bm z})}\lesssim   &\   \int     |I \varphi_0|+ | I\varphi_2|  \dd x
 &&
   \text{  with }
   &&
   I  \defeq    f(Q+ \tilde{\boldsymbol{\Phi}}[{\bm z}]_1)   - f(Q)-f' (Q)  \tilde{\boldsymbol{\Phi}}[{\bm z}]_1     . \label{eq:defI}
\end{align}
Notice that
\begin{align} \label{eq:esttildeph}
  \left |   \tilde{\boldsymbol{\Phi}}[{\bm z}]_1\right | \lesssim |z_1| \ |\varphi _0|  + |z_2| \ |\varphi _2| ,
\end{align}
 with $|\varphi_0|\sim e ^{-\frac{p+1}{2}|x|}$ and $|\varphi_2|\lesssim e ^{-k_2|x|}=e ^{-     \frac{3-p}{2}|x|} $ by \cref{eq:asympt_eigenf}.
  Now we estimate pointwise $I $.

\begin{case}\label{case:one}
Let $|\tilde{\boldsymbol{\Phi}}[{\bm z}]_1| \leq  Q/2$.   By $Q\sim e^{-|x|}$,  $k_2=\frac{3-p}{2} $ and $ \av{f''\pare{Q+\tau \tilde{\bm{\Phi}}\bra{{\bm z}}_1 }} \sim Q^{p-2} $, for $ \tau\in \bra{0, 1} $ we have
\begin{multline}
 |I|  \lesssim  \sup_{\tau\in\bra{0, 1}}\av{f''\pare{Q+\tau \tilde{\bm{\Phi}}\bra{{\bm z}}_1 }} \  \pare{  |z_1|^2 |\varphi_0|^2 + |z_2|^2 |\varphi_2|^2 }   \lesssim  |Q|^{p-2} \pare{ |z_1|^2 |\varphi_0|^2 + |z_2|^2 |\varphi_2|^2 }    \\
  \lesssim  Q ^{p-2} |{\bm z}|^2 e^{-2k_2|x|}\sim  |{\bm z}|^2 e^{|x| \pare{  2-p - (3-p) }  } = |{\bm z}|^2 e^{-|x| } .\label{eq:Ijest}
\end{multline}
\end{case}

\begin{case}\label{case:two}
Let $ |\tilde{\boldsymbol{\Phi}}[{\bm z}]_1| >  Q/2 $.   Then
\begin{align}\label{eq:dicot1}
    |z_1| \ |\varphi_0|  + |z_2| \ |\varphi_2| \gtrsim |\tilde{\boldsymbol{\Phi}}[{\bm z}]_1|\gtrsim Q && \Longrightarrow &&  |z_2| \ |\varphi_2|  \gtrsim Q
\end{align}
  by $|\varphi_0|\sim e ^{-\frac{p+1}{2}|x|}$, $Q\sim e ^{- |x|} $ and $|z_1|\lesssim \delta \ll 1$, which obviously yield  $ |z_1| \ |\varphi_0|\ll Q$  always. Then
\begin{equation}
\begin{aligned}
   |I|\lesssim   \left |  \tilde{\boldsymbol{\Phi}}[{\bm z}]_1  \right  |^p         \lesssim |z_2 \varphi_2| ^p   \lesssim |z_2|^p e ^{- p \frac{3-p}{2} |x|} &=   |z_2|^p   e^{ - (2-p) \frac{ p-1}{2}  |x|}        e ^{ 2^{-1}\left( (2-p)  ( p-1)- p (3-p)    \right)|x|}
    \\&=   |z_2|^p   e^{ - (2-p) \frac{ p-1}{2}  |x|}      e^{-|x| } .  \label{eq:Ijest1}
\end{aligned}
\end{equation}
It is easy to conclude from \eqref{eq:dicot1}, that $|z_2| e ^{-\frac{3-p}{2}|x|} \gtrsim e^{-|x|}$, and so $|z_2|  \gtrsim e^{ -   \frac{ p-1}{2}  |x|}$ or $|z_2| ^{2-p} \gtrsim e^{ - (2-p) \frac{ p-1}{2}  |x|}$. This proves that $ \av{I}\lesssim \av{{\bm z}}^2 e^{-\av{x}} $ when $ |\tilde{\boldsymbol{\Phi}}[{\bm z}]_1| >  Q/2 $.
\end{case}
 Cases \ref{case:one} and \ref{case:two}  yield
\begin{align}
   \av{I}\lesssim      |{\bm z}|^2    e^{-|x| }  \text{ for all $x\in \R$}.  \label{eq:Ijest2}
\end{align}
     {  This yields    \eqref{item:rem_better_1}
and, by dominated convergence,   $\tilde{{\bm z}}_R \in C^0 \pare{ D _{\C^2} (0, \delta _1), \C^2 }$, proving \Cref{item:rem_better_20}.}  \\

We prove \Cref{item:rem_better_1.1}. It suffices to prove that the right hand side in \eqref{R:orth1} is in $  C^1 \pare{ D _{\C^2} (0, \delta _1), \C^2 }$. {  We have using \cref{eq:DzPhi,eq:errp1}}
\begin{equation}
\label{eq:first_differential_hatR}
D_{{\bm z}}\hat{ {\bm R}} \pare{{\bm z}} \ \bm{\zeta} = {  \pare{f'\pare{\bm{\Phi}({\bm z})_1} - f'\pare{Q}} \tilde{\bm \Phi}\bra{\bm \zeta }_1 \ \mathbf{j}}.
\end{equation}
 For $ \av{\tilde{\boldsymbol{\Phi}}[{\bm z}]_1} \ll  Q $ we have
\begin{align}\label{eq:estder1}
  \av{f' (Q)   -  f'(\boldsymbol{\Phi}\pare{\bm z}_1)} \lesssim \left |f'' (Q)   \tilde{\boldsymbol{\Phi}}[{\bm z}]_1 \right |  \lesssim  |{\bm z}| e^{|x| \left( 2-p- \frac{3-p}{2}\right) } =  |{\bm z}| e^{-\frac{p-1}{2}|x|}
\end{align}
while  for  $ |\tilde{\boldsymbol{\Phi}}[{\bm z}]_1|\gtrsim   Q $, by $|z_2|\gtrsim e^{- \frac{p-1}{2} |x|}$, and so, by   $|z_2| ^{2-p} \gtrsim e^{ - (2-p) \frac{ p-1}{2}  |x|}$,
 we have
\begin{align}\label{eq:estder2}
  |f' (Q)   -  f'(\boldsymbol{\Phi}\pare{\bm z}_1)| &\lesssim    |     \tilde{\boldsymbol{\Phi}}[{\bm z}]_1   | ^{p-1} \lesssim |z_2|^{p-1}
   e^{-(3-p)\frac{p-1}{2}|x|} =|z_2|^{p-1}
   e^{-(2-p)\frac{p-1}{2}|x|}    e^{- \frac{p-1}{2}|x|} \\& \lesssim  |{\bm z}| e^{-\frac{p-1}{2}|x|}.\nonumber
\end{align}
So, since the differentiation in \cref{eq:first_differential_hatR} is well defined and is uniformly bounded by an integrable function,
    by dominated convergence   we can commute derivative and integral and conclude
\begin{align*}
 D_{{\bm z}}\< \mathbf{J} \hat{{\bm R}}\pare{\bm z}   , D _{{\bm z}}\boldsymbol{\Phi}\pare{\bm z}   \Theta  \> \bm{\zeta} = \< \mathbf{J} D_{\bm z}\hat{{\bm R}}\pare{\bm z}  \bm{\zeta} , D _{{\bm z}}\boldsymbol{\Phi}\pare{\bm z}   \Theta  \>
\end{align*}
with, furthermore, the differential continuous in ${\bm z}$. So, from  \cref{R:orth1} we obtain
  $\tilde{{\bm z}}_R \in C^1 \pare{  D _{\C^2} (0, \delta _1), \C^2 }$.\\

{  We prove  \Cref{item:rem_better_2}.    From \cref{eq:first_differential_hatR} and the linearity in ${\bm z}$ of $ \tilde{ \bmPhi}\bra{{\bm z} }$ ,
\begin{equation}
 D^2_{\bm z} \hat{\bm R}\pare{\bm z} \pare{{\bm \zeta}^1, {\bm \zeta}^2 }
 =-
 f''\pare{\bmPhi \pare{\bmz}_1}  \tilde{ \bmPhi}\bra{{\bm \zeta}^1}_1  \   \tilde{\bmPhi}\bra{{\bm \zeta}^2}  _1    \mathbf{ j} .\label{2ndderhatr}
\end{equation}
%so that, by $Q(x) \sim e^{-|x|}$, \eqref{eq:valk}  and \eqref{eq:asympt_eigenf}, we have \begin{align*}\av{D^2_{\bmz} \hat {\bm R}\bra{0} \pare{{\bm \zeta}^1, {\bm \zeta}^2 }}&\lesssim \av{f''\pare{Q}} \pare{\av{\varphi_0}^2 + \av{\varphi_2}^2}  \av{\bm \zeta^1}\av{\bm \zeta^2}  \lesssim e^{-|x|} \av{\bm \zeta^1}\av{\bm \zeta^2}.\end{align*}
Now, for fixed $A=R,I$, $z_{a R}:=\Re z_a$, $z_{a I}:=\Im z_a$ and $a=1,2$, we have
\begin{align}\label{eq:2diff}& \< \mathbf{J}  \( \partial _{z_{a A}}\hat{{{\bm R}}}\pare{ {\bm z} } -  D_\bmz  \partial _{z_{a A}}\hat{{{\bm R}}}\pare{0}  \bmz   \)    , D _{{\bm z}}\boldsymbol{\Phi}\pare{ {\bm z} }   \Theta  \>   =
- \<   II ,1 \>  \\&  \text{with }II:= \( f'(\boldsymbol{\Phi}\pare{\bm z}_1) - f'(Q)- f''(Q)  \tilde{\boldsymbol{\Phi}}[{\bm z}] _1          \)    \partial _{z_{a A}}\tilde{\boldsymbol{\Phi}}[{\bm z}]_1      \overline{D _{{\bm z}}\boldsymbol{\Phi}\pare{ {\bm z} } _1  \Theta }.\nonumber
\end{align}
To show that the above expression is $o(|{\bm z}|)$ we bound $II$. For   $ |\tilde{\boldsymbol{\Phi}}[{\bm z}]_1|\ll  Q $, we have
\begin{align*} &
  |II| \lesssim  \left |f'''(Q) \right | \  \left | \tilde{\boldsymbol{\Phi}}[{\bm z}] _1   e^{- k_2|x|} \right | ^{2} \lesssim
  |{\bm z}|^{2} e^{ \( 3-p - 2 (3-p)\)   |x| } =  |{\bm z}|^{2} e^{-   ( 3-p   )  |x|  } .
\end{align*}
  For $ \av{\tilde{\boldsymbol{\Phi}}[{\bm z}]_1} \gtrsim  Q $,
like in \eqref{eq:estder2} we have  %, which completes the proof of the 2nd order differentiability of $\tilde{{\bm z}} _R$ at ${\bm z}=0$,
\begin{align*}
  |II| &\lesssim     \left | \tilde{\boldsymbol{\Phi}}[{\bm z}] _1   \right | ^{p-1} e^{- 2k_2|x|}\sim   \left | z_2 \psi _2   \right | ^{p-1} e^{- 2k_2|x|} \lesssim   |z_2 | ^{p  }   |z_2 | ^{ -1 } e ^{-\frac{(3-p) ^2}{2}|x|}\\&  \lesssim  | {\bm z} | ^{p  }  e^{  \frac{1}{2}\(  {p-1}   -(3-p) ^2 \) |x|} \le  | {\bm z} | ^{p  } e^{-|x|} .
   \end{align*}
So  we can use dominated convergence obtaining
\begin{align*}
   \lim _{\bm z\to 0}   \frac{1}{\av{\bm z}} \< \mathbf{J}  \( \partial _{z_{a A}}\hat{{{\bm R}}}\pare{\bm z} -  D_\bmz  \partial _{z_{a A}}\hat{{{\bm R}}}\pare{0}  \bmz   \)    , D _{{\bm z}}\boldsymbol{\Phi}\pare{\bm z}   \Theta  \> =    \lim _{\bm z\to 0}   O\(  | {\bm z} | ^{p -1 }\) =0    .
\end{align*}
This shows that  the function in \eqref{eq:2diff} is differentiable in $\bm z=0$ for any $a=1,2$ and any $A=R,I$ or, equivalently, that the right hand side in \eqref{R:orth1}, and so also $\tilde{{\bm z}}_R $, is twice differentiable in $\bm z=0$. }

   \end{proof}

We set now
\begin{align}\label{eq:tildez}& \tilde{\bm z}\defeq  \tilde{{\bm z}}_0 + \tilde{{\bm z}}_R
\end{align}
and we write part of the square component of the Taylor expansion at 0 of $ {\bm R}\pare{\bm z}$ considering
\begin{align}\label{eq:G0}  z_2^2\mathbf{G}_2 (x)+ \overline{z}_2^2\overline{\mathbf{G}}_2 (x) &&  \text{  where }
&&
\mathbf{G}_2:=  D_{\bm z}\boldsymbol{\Phi}\pare{\bm z} \left . \partial ^2_{z_2}\tilde{{\bm z}}_R \right |_{{\bm z}=0} -f''\pare{Q} \varphi _{ 2}^2 \mathbf{j},
\end{align}
where, using the notation above \cref{symplmatr} we  applied $ \partial  _{z_2 } := 2^{-1}\(  \partial _{z_{2R}} - \im  \partial _{z_{2I}}   \)$ and we used  \cref{2ndderhatr} with ${\bm \zeta}^1 = {\bm \zeta}^2 = \( 0,\dfrac{1-\im}{2} \) ^\intercal $ to get
\begin{equation*}
\left .  \partial ^2_{z_2}\hat{\bm R}\pare{\bm z}  \right | _{{\bm z} =0} = -f''\pare{Q} \varphi _{ 2}^2 \mathbf{j} .
\end{equation*}
We remark that $f''\pare{Q} \varphi _2^2 =-p (p-1) Q ^{p-2} \varphi _2^2 \sim  -e^{ -       |x|      }$ for $x\to \infty$.

\begin{rem}
  Notice that   $\left . \partial ^2_{z_2}\tilde{{\bm z}}_R \right |_{{\bm z}=0}$ is well defined
by the double differentiability of $\tilde{{\bm z}}_R$ in ${\bm z}=0$, but its exact formula is immaterial because its contribution to the Fermi Golden Rule is null  and its only purpose is to give
the orthogonality in \eqref{R:orth}, which in turn is quite useful later in \Cref{lem:modbound}. For the sake of completeness we show how to compute $ \tilde{\bm z}_R $ which in turn explicit the expression in \eqref{eq:G0}:
by the definition  of $\boldsymbol{\Phi}\pare{\bm z}$ in  { \cref{eq:rp,eq:DzPhi}}, the left hand side  in \eqref{R:orth1} is constant in ${\bm z}$.  in particular  the l.h.s. of \eqref{R:orth1} is equal to
\begin{footnotesize}
\begin{multline}\label{eq:system_ztildeR}
\<   \mathbf{J} D _{{\bm z}}\boldsymbol{\Phi}\pare{\bm z}  \tilde{{\bm z}}_R  , D _{{\bm z}}\boldsymbol{\Phi}\pare{\bm z}   \Theta  \>
=
 \
\pare{
2\Re\pare{\tilde{z}_{R,1} {\bf J \Phi_0} + \tilde{z}_{R,2} {\bf J \Phi_2}}
\ , \
2\Re\pare{\Theta_1 {\bf  \Phi_0} + \Theta_2 {\bf \Phi_2}}
} \\
%------------------------
= 4\pare{
\Re \tilde{z}_{R, 1}  \frac{{\bf J} \bf Y_+}{2}
+ \Im \tilde{z}_{R, 1}  \frac{{\bf J}\bf Y_-}{2}
- \Re \tilde{z}_{R, 2} {\bf j} \varphi_2
- \Im \tilde{z}_{R, 2}\lambda {\bf i} \varphi_2
\ \middle| \
\Re \Theta_1  \frac{ \bf Y_+}{2}
+ \Im \Theta_1  \frac{\bf Y_-}{2}
+ \Re \Theta_2  {\bf i} \varphi_2
- \Im \Theta_2 \lambda {\bf j} \varphi_2
} .
\end{multline}
\end{footnotesize}
Letting $ {\bf \Theta}\in \set{\pare{1, 0}, \pare{\ii, 0}, \pare{0, 1}, \pare{0, \ii}} $ we obtain (cf. \cref{eq:matr_eig1,eq:errp1,eq:eigC}) 
\begin{equation}
\label{eq:ztilde_expression}
\begin{footnotesize}
\begin{aligned}
\Re \tilde{z}_{R, 1} = & -\frac{1}{2\nu_0 \norm{\varphi_0}_{L^2}^2} \ps{{\bf J}\hat{\bm R} \pare{\bm z} }{\bf Y_-}
&
\Im \tilde{z}_{R, 1} = & \frac{1}{2\nu_0 \norm{\varphi_0}_{L^2}^2} \ps{{\bf J}\hat{\bm R} \pare{\bm z} }{\bf Y_+}
\\
\Re \tilde{z}_{R, 2} = & \frac{1}{4\lambda \norm{\varphi_2}_{L^2}^2} \ps{{\bf J}\hat{\bm R} \pare{\bm z} }{\binom{0}{-2\lambda\varphi_2}}=0
&
\Im \tilde{z}_{R, 2} = & \frac{1}{4\lambda \norm{\varphi_2}_{L^2}^2} \ps{{\bf J}\hat{\bm R} \pare{\bm z} }{\binom{2\varphi_2}{0}}
\end{aligned}
\end{footnotesize} \ .
\end{equation}
\end{rem}

\begin{remark}[Fermi Golden Rule]\label{rem:FGR}
Notice that the existence of a function $\mathbf{g}_{2}\in L^\infty\pare{\bR} $ which is a solution  of $\mathbf{J}\mathbf{L}_0 \mathbf{g}_{2}=\im 2\lambda  \mathbf{g}_{2}$ s.t.
\begin{align}\label{ass:FGR1}
\< \mathbf{J} \mathbf{G}_{2},\mathbf{g}_{2}\> \neq 0,
\end{align}
is equivalent to \Cref{ass:FGRintro} since, $ \mathbf{g}_{2} =   {  ( 1, 2 \im \lambda ) ^ \intercal } \  c  \ g$ with $ L_0g= 4\lambda ^2 g$ and $c\in \C \backslash \{  0\}$ and the definition of $\boldsymbol{\Phi}\pare{\bm z}$  gives a cancellation and  leads  to \begin{align*}
  \< \mathbf{J} \mathbf{G}_{2},\mathbf{g}_{2}\> =  - \overline{c}\ p\ (p-1) \<Q ^{p-2}\varphi _2^2  ,g \> .
\end{align*}
where we used
\begin{align*}
\< \mathbf{J} D_{\bm z}\boldsymbol{\Phi}\pare{\bm z} \boldsymbol{\Theta},\mathbf{g}_{2}\>  = 0, \text{ for any $\boldsymbol{\Theta}\in \C^2$}
\end{align*}
by the well known,  elementary to check by a simple integration by parts, simplectic orthogonality between discrete and continuous modes  of  $\mathbf{J}\mathbf{L}_0$.
\end{remark}
We end this section observing that from \eqref{eq:rp1} and \eqref{eq:errp2} we have
\begin{align}\label{eq:rp2}
	D_{{\bm z}}\boldsymbol{\Phi}\pare{\bm z}\tilde{{\bm z}}= \mathbf{J}
	\begin{pmatrix}
		-\partial_x^2+1 & 0\\ 0 &1
	\end{pmatrix}
	\boldsymbol{\Phi}\pare{\bm z} + f(\boldsymbol{\Phi}\pare{\bm z}_1) \mathbf{j} + {\bm R}\pare{\bm z}.
\end{align}

\section{Main estimates and proof of Theorem \ref{thm:main1}.} \label{sec:modul}

Consider a solution ${\bm u}\in C^0 \pare{ [0,+\infty ) ; \boldsymbol{\mathcal{H}}^{1}  _{\even} } $  gravitating around $Q\mathbf{i}$  like in \eqref{eq:cond_sta}.
From the spectral decomposition of ${\bm u}-Q\mathbf{i}$, we have ${\bm u}=\boldsymbol{\Phi}\pare{\bm z}+\boldsymbol{\eta} $ with   $\boldsymbol{\eta}\in C^0([0,+\infty ) , \boldsymbol{\mathcal{H}}^{1} _{c} )$ and  the NLKG \eqref{NLKG} rewrites  \begin{align}\nonumber
\dot{\boldsymbol{\eta}}+D_{{\bm z}}\boldsymbol{\Phi}\pare{\bm z}(\dot{{\bm z}}-\tilde{{\bm z}})+\xcancel{ D_{{\bm z}}\boldsymbol{\Phi}\pare{\bm z} \tilde{{\bm z}}} =\mathbf{J}
 \mathbf{L}_0\boldsymbol{\eta} +\xcancel{ \mathbf{J}
 \mathbf{L}_0\boldsymbol{\Phi}\pare{\bm z}} + \xcancel{ f(\boldsymbol{\Phi}\pare{\bm z}_1) \mathbf{j}} + \xcancel{{\bm R}\pare{\bm z} }
  - {\bm R}\pare{\bm z} \\   + \underbrace{\left(  f'\pare{ \boldsymbol{\Phi}\pare{\bm z}_1 } - f'(Q) \right) \eta _1 \mathbf{j}  + \left( f\left( \boldsymbol{\Phi}\pare{\bm z}_1 +\eta _1 \right)-f(\boldsymbol{\Phi}\pare{\bm z}_1)- f'(\boldsymbol{\Phi}\pare{\bm z}_1)\eta _1  \right) \mathbf{j}} _{\defeq \bm{F}_Q\pare{ \bmz, \bm{\eta}}}  ,\label{eq:modeq}
\end{align}
for ${\bm R}\pare{\bm z}$ defined in \eqref{eq:errp1} and \eqref{eq:errp2} and where the canceling follows from \eqref{eq:errp1}.
Equivalently, we have also
\begin{equation}\label{nuovaequi}
	\dot{\boldsymbol{\eta}}+D_{{\bm z}}\boldsymbol{\Phi}\pare{\bm z}(\dot{{\bm z}}-\tilde{{\bm z}})
	=\mathbf{J}
	\begin{pmatrix}
		-\partial_{x}^2+ 1 &0\\0&1
	\end{pmatrix}\boldsymbol{\eta}
	- {\bm R}\pare{\bm z}
	 + \left( f\left( \boldsymbol{\Phi}\pare{\bm z}_1 +\eta _1 \right)-f(\boldsymbol{\Phi}\pare{\bm z}_1) \right) \mathbf{j}.
\end{equation}
 In the sequel we will consider constants  $A, B,\varepsilon >0$ satisfying
 \begin{align}\label{eq:relABg}
\log(\delta ^{-1})\gg\log\pare{ \epsilon ^{-1} } \gg  A\gg    B^2\gg B \gg  \exp \left( \varepsilon ^{-1} \right) \gg 1.
 \end{align}
We will fix $A,B$ and $\varepsilon$ satisfying the above relation always retaking $A,B$ larger and $\varepsilon$ smaller if necessary.
Then, we will take $\epsilon$ sufficiently small and finally chose $\delta$ sufficiently small.
 Exploiting  the double differentiability of $\tilde{{\bm z}}_R$ in ${\bm z}=0$, the choices of the constants are made so that
\begin{align}\label{eq:142bis}
   A  \av{\tilde \bmz _R\pare{\bmz} - \frac{D_\bmz^2 \tilde\bmz_R\pare{0}\pare{\bmz, \bmz}}{2} } =  o\pare{ \delta ^2 }  \text{  for } \av{{\bm z}} \lesssim \delta.
\end{align}
   We will denote by    $o_{\alpha}(1)$  constants depending on $\alpha >0$ such that
 \begin{align}\label{eq:smallo}
o_{\alpha}(1) \xrightarrow {\alpha  \to 0^+   }0.
 \end{align}
 We will consider the norms
\begin{align}\label{eq:normA}&
\norm{ \boldsymbol{\eta} }_{ \boldsymbol{ \Sigma }_A} \defeq\left \| \sech \pare{\frac{2}{A} x} \eta_1'\right \|_{L^2} +A^{-1}\left \|    \sech \pare{\frac{2}{A} x} \boldsymbol{\eta}\right\|_{L^2}  \\
%-----------------------------
&  \norm{ \boldsymbol{\eta} }_{ \boldsymbol{L }^2_{-\kappa} } \defeq\left \| \sech \pare{ \kappa  x}  \boldsymbol{\eta}\right \|_{L^2}.\label{eq:normk}
\end{align}
We will prove the following  continuation argument.
 \begin{proposition}\label{prop:contreform}   Under the Assumption  \ref{rem:FGR},
  for any small $\epsilon>0 $
there exists  a    $\delta _0 \defeq \delta _0(\epsilon )  \ll \epsilon $ such that if a solution of \eqref{NLKG} satisfying \eqref{eq:cond_sta},
if   in  $I=[0,T]$ we have
\begin{align}&
\|\dot {{\bm z}} - \tilde{{\bm z}}\|_{L^2(I)}^2+\| {z} _1\|_{L^2(I)}^2 +\| {z} _2\|_{L^4(I)} ^4+ \|   \boldsymbol{\eta}  \|_{L^2(I, \boldsymbol{ \Sigma }_A  \cap   \boldsymbol{L }^2_{-\kappa})}^2\le   \epsilon ^2  \label{eq:main11}
\end{align}
then  for $\delta  \in (0, \delta _0)$
     inequality   \eqref{eq:main11} holds   for   $\epsilon^2$ replaced by $ o_{\varepsilon}(1)   \epsilon ^2$    where $o_{\varepsilon}(1) \xrightarrow {\varepsilon  \to 0^+   }0 $.
\end{proposition}
% Notice that Proposition \ref{prop:contreform} implies by standard continuation arguments Theorem \ref{thm:main}.
It is elementary that     Proposition \ref{prop:contreform} follows  from the following
  \Cref{prop:modp,prop:FGR,prop:1stvirial,prop:2ndvirial}, where we always assume that our solution satisfies the stability condition \eqref{eq:cond_sta}.

\begin{proposition}\label{prop:modp}
 Under the assumptions of \Cref{prop:contreform} there exists a $ \kappa > 0 $ such that
\begin{align}
\|\dot{{\bm z}}-\tilde{{\bm z}}\|_{L^2(I)}^2=
\delta^{2(p-1)} \norm{  \boldsymbol{\eta}  }_{L^2 \pare{ I ;   \boldsymbol{L }^2_{-\kappa} } }^2      . \label{eq:lem:estdtz}
\end{align}
\end{proposition}

 \begin{proposition}[Radiation Damping] \label{prop:FGR}
We have \begin{align}\label{eq:FGRint}
 \norm{ {z} _1}_{L^2(I)}^2 +\norm{ {z} _2 }_{L^4(I)} ^4=  o(\epsilon ^2) .
 \end{align}
 \end{proposition}

\begin{proposition}[1st virial estimate]\label{prop:1stvirial}
We have
\begin{align}
 \|   \boldsymbol{\eta}  \|_{L^2(I, \boldsymbol{ \Sigma }_A  )}^2  \lesssim \delta^2 +  \|  \boldsymbol{\eta}  \|_{L^2(I,  \boldsymbol{L }^2_{-\kappa})}^2  + \| \mathbf{{z} } \|_{L^4(I)} ^4.\label{eq:1stvInt}
\end{align}

\end{proposition}

\begin{proposition}[2nd virial  estimate]\label{prop:2ndvirial}
We have
\begin{align}\label{eq:2ndv}
 \|  \boldsymbol{\eta}  \|_{L^2(I,  \boldsymbol{L }^2_{-\kappa})}^2  \lesssim  {\delta}^2 +o _\varepsilon (1)  \|   \boldsymbol{\eta}  \|_{L^2(I, \boldsymbol{ \Sigma }_A  )}^2  +\| {\bm z}  \|_{L^4(I)}^4   .
\end{align}
\end{proposition}

  \begin{proof}[Proof of Theorem \ref{thm:main1}]
By continuity,  Proposition \ref{prop:contreform} implies that inequality \eqref{eq:main11} is valid with $I=\R _+$. This implies, adjusting $\epsilon$,  inequality  \eqref{eq:main20} for $I=\R _+$.
\begin{align}\label{stimello}
 \int _{\R }  \|     \boldsymbol{\eta }  (t) \| ^2 _{\boldsymbol{\mathcal{H}}^{1} _{-a}(\R )} \dd t +  \av{\bm z}^4 _{L^4 (\R  )}   \le \epsilon^2
  .
\end{align}
  From  \eqref{eq:cond_sta},  \eqref{eq:tildez} and  \eqref{lem:modbound1} below,   we have $\dot {{\bm z}}\in L^\infty (\R , \C ^2 )$ that, with  \eqref{eq:main20} implies \eqref{eq:main3}. To prove \eqref{eq:main2}, let
  \begin{equation*}
  \begin{aligned}
  \mathsf{a} \defeq \norm{\bm \eta }^2_{\cH^1_{-a}}
  &\ =
  \psc{\sech^2\pare{a x}\eta_1'}{\eta_1'}
  +
  \psc{\sech^2\pare{a x}\eta_1}{\eta_1}
  +
  \psc{\sech^2\pare{a x}\eta_2}{\eta_2}
  \\
  & \ \eqdef \ \mathsf{a}_1\pare{t} + \mathsf{a}_2\pare{t} + \mathsf{a}_3 \pare{t}.
  \end{aligned}
  \end{equation*}
  Using \cref{eq:DzPhi,nuovaequi} we have
  \begin{equation*}
  \dot{\mathsf{a}}_2 = \psc{\sech^2\pare{ax}\eta_1}{- \tilde {\bm \Phi}\bra{\dot{{\bm z}}-\tilde{{\bm z}}}_1
	+
	\eta_2
	- {\bm R}\pare{\bm z}_1
	 }.
  \end{equation*}
  By \cref{lem:modbound1,eq:errp2,eq:errp1,item:rem_better_1}   we obtain
  \begin{equation*}
  \av{\dot{\mathsf{a}}_2}\lesssim \delta^2.
  \end{equation*}
  Proceeding  similarly,
  \begin{equation}\label{eq:timeder}
  \begin{aligned}
  {\dot{\mathsf{a}}_1} = & \ - 2 \psc{\sech^2 \pare{ax}\eta_1''}{\eta_2} + \cO\pare{\delta^2}
  \\
    {\dot{\mathsf{a}}_3} = & \  2 \psc{\sech^2\pare{ax}\eta_1''}{\eta_2} + \cO\pare{\delta^2}
  \end{aligned}
  \end{equation}
where, in the sum of the two terms, the explicit terms in the right hand side  cancel out,
so that we obtain $ \av{\dot{\mathsf{a}}} =\mathcal{O}\pare{\delta^2} $. This and  \eqref{eq:main20} yield \eqref{eq:main2}. Notice that the computations in \eqref{eq:timeder} are formal, but they can be made rigorous, defining $ \mathsf{a}_{\varepsilon} \defeq \norm{  \< \varepsilon \im \partial _x\> ^{-1}  \bm \eta }^2_{\cH^1_{-a}}$ proving $ \av{\dot{\mathsf{a}} _{\varepsilon} }\le C \delta^2$ with a constant $C>0$ independent from $\varepsilon$ and then letting $\varepsilon \to 0^+ $  to get the bound $ \av{\dot{\mathsf{a}}   }\le C \delta^2$. Similar regularizations can be used to justify rigorously some of the computations later in the manuscript.
\end{proof}

 \section{Proof of Proposition \ref{prop:modp}: bounds for the approximate time-derivative of the discrete modes  }
\label{sec:propdmodes1}

\begin{notation}\label{notation:chi}
We fix an even function $\chi\in C_0^\infty(\R , [0,1])$ satisfying
$1_{[-1,1]}\leq \chi \leq 1_{[-2,2]}$ and $x\chi'(x)\leq 0$ and set $\chi_A\defeq\chi(\cdot/A)$.
\end{notation}

\begin{lemma}\label{lem:estF}
For $ \delta > 0 $, $ \bmz \in D_{\mathbb{C}^2}\pare{0, \delta}$ and  $ \bm{\eta} \in D_{ \cH^1_c}\pare{0, \delta} $   then, for any $ A> 0 $ and $ \kappa >0$,
\begin{align}   \label{eq:lem:estF1}
\left \|    \sech \pare{  \kappa x}    \pare{  f'(\boldsymbol{\Phi}\pare{\bm z}_1) - f'(Q) } \eta _1\right \|_{L^2} \lesssim  \delta ^{p-1} \left \|    \sech \pare{  \pare{\kappa +  (p-1)k _2} x} \eta_1\right \| _{L^2}, \\
%----------------------------------------------------------------
\norm{\chi_A  \pare{  f'(\boldsymbol{\Phi}\pare{\bm z}_1) - f'(Q) } \eta _1}_{L^1} \lesssim {  \delta^{p-1}} \norm{      \sech \pare{   \frac{2}{A}x}     \eta_1}_{L^2},    \label{eq:lem:estF2}\\
%----------------------------------------------------------------
 \norm{   \sech \pare{  \kappa x}    \pare{ f\pare{ \boldsymbol{\Phi}\pare{\bm z}_1 +\eta _1 }-f(\boldsymbol{\Phi}\pare{\bm z}_1)- f'(\boldsymbol{\Phi}\pare{\bm z}_1)\eta _1  }}_{L^2}
  \lesssim  \delta ^{p-1} \norm{        \sech \pare{  \kappa x}    \eta_1}_{L^2}, \label{eq:lem:estF3} \\
 %----------------------------------------------------------------
\norm{ \chi_A \pare{ f\pare{ \boldsymbol{\Phi}\pare{\bm z}_1 +\eta _1 }-f(\boldsymbol{\Phi}\pare{\bm z}_1)- f'(\boldsymbol{\Phi}\pare{\bm z}_1)\eta _1  } }_{L^1} \lesssim A^{1-\tfrac{p}{2}} { \delta ^{p-1}}\norm{     \sech \pare{   \frac{2}{A}x}     \eta_1 }_{L^2}.    \label{eq:lem:estF4}
\end{align}
\end{lemma}
\begin{proof}
\begin{step}[Proof of \eqref{eq:lem:estF1} and \eqref{eq:lem:estF2}]
First of all in the points where $|\tilde{\boldsymbol{\Phi}}[{\bm z}]_1|\ll Q$,
\begin{align*}
  \left | f'(\boldsymbol{\Phi}\pare{\bm z}_1) - f'(Q) \right | \lesssim  \left |    f''(Q)\tilde{ \boldsymbol{\Phi}}[{\bm z}]_1 \right | \sim Q ^{p-2}\left |     \tilde{ \boldsymbol{\Phi}}[{\bm z}]_1 \right | ^{1-(2-p)+(2-p)}\le \left |     \tilde{ \boldsymbol{\Phi}}[{\bm z}]_1 \right | ^{p-1}.
\end{align*}
If instead $|\tilde{\boldsymbol{\Phi}}[{\bm z}]_1|\gtrsim Q$ then
\begin{align*}
  \left | f'(\boldsymbol{\Phi}\pare{\bm z}_1) - f'(Q) \right | \lesssim  \left |    f'  \pare{\tilde{ \boldsymbol{\Phi}}[{\bm z}]_1  }\right |
  \sim   \left |     \tilde{ \boldsymbol{\Phi}}[{\bm z}]_1 \right | ^{p-1}.
\end{align*}
Thus, by \eqref{eq:asympt_eigenf} we have
\begin{align}\label{eq:pointwise_bound}
  \left |  \pare{  f'(\boldsymbol{\Phi}\pare{\bm z}_1) - f'(Q) } \eta _1 \right | \lesssim  \left |    {\bm z} \right | ^{p-1}  \sech \pare{ (p-1) k _2 x}  |\eta _1|  ,
\end{align}
which yields \eqref{eq:lem:estF1}. We prove now \eqref{eq:lem:estF2}. Using \eqref{eq:pointwise_bound} and   $ \sech\pare{\frac{2}{A} x} \sim 1 $ on $ \bra{-A, A} $,
\begin{align*}
\norm{\chi_A  \pare{  f'(\boldsymbol{\Phi}\pare{\bm z}_1) - f'(Q) } \eta _1}_{L^1}\lesssim & \  \av{\bm z}  ^{p-1}\int _{-A}^{A}      \sech \pare{ (p-1) k _2 x}  |\eta _1\pare{x}|  \dd x
\\
\lesssim  & \ \delta^{p-1} \norm{\sech\pare{\frac{2}{A} \ x} \eta_1}_{L^2}.
\end{align*}
\end{step}

 \begin{step}[Proof of \eqref{eq:lem:estF3} and \eqref{eq:lem:estF4}]
If $|\eta _1|\gtrsim |\boldsymbol{\Phi}\pare{\bm z}_1|$  we have
\begin{align}\label{eq:buuuuh1}
 \left |  f\pare{ \boldsymbol{\Phi}\pare{\bm z}_1 +\eta _1 }-f(\boldsymbol{\Phi}\pare{\bm z}_1)- f'(\boldsymbol{\Phi}\pare{\bm z}_1)\eta _1  \right | \lesssim | \eta _1 | ^{p}.
\end{align}
In the points where $|\eta _1|\ll |\boldsymbol{\Phi}\pare{\bm z}_1|$, for some $t \in (0,1)$ we have
\begin{align} \label{eq:buuuuh2}
 \left |  f\pare{ \boldsymbol{\Phi}\pare{\bm z}_1 +\eta _1 }-f(\boldsymbol{\Phi}\pare{\bm z}_1)- f'(\boldsymbol{\Phi}\pare{\bm z}_1)\eta _1  \right | = \left |  \pare{f'\pare{ \boldsymbol{\Phi}\pare{\bm z}_1 +t\eta _1 } - f'(\boldsymbol{\Phi}\pare{\bm z}_1) } \eta _1  \right |
 %--------------------------------------
 \\ \lesssim  \left |   f''\pare{ \boldsymbol{\Phi}\pare{\bm z}_1   }   \eta _1 ^2  \right | \sim |\boldsymbol{\Phi}\pare{\bm z}_1 | ^{p-2}   \eta _1 ^2 \le  | \eta _1 | ^{p}. \nonumber
\end{align}
Inequality \eqref{eq:lem:estF3}  follows immediately.
 By  $  \sech \pare{   \frac{2}{A}x} \sim 1$ in $\mathrm{supp} \chi_A$,  \eqref{eq:cond_sta} and H\"older inequality, we get
\begin{align*}
\|\chi_A | \eta _1 | ^{p}\|_{L^1}& \ \lesssim  A^{1-\tfrac{p}{2}}  \norm{\sech\pare{\frac{2}{A} \ x} \eta_1}_{L^{ \frac{2}{p}}}   \lesssim  A^{1-\tfrac{p}{2}} \delta ^{p-1} \norm{\sech\pare{\frac{2}{A} \ x} \eta_1}_{L^2},
\end{align*}
which proves \eqref{eq:lem:estF4}.
 \end{step}
\end{proof}

\begin{lemma}
\label{lem:modbound}
With the same hypothesis of \Cref{lem:estF} and for any $\kappa\in\pare{0, k_2} $ the following bound holds true
\begin{align}\label{lem:modbound1}
\av{ \dot{{\bm z}}\pare{t}-\tilde{{\bm z}}\pare{t}}   \lesssim  \delta ^{p-1}   \norm{     \sech \pare{  \kappa x} \boldsymbol{\eta}\pare{t} } _{L^2\pare{\mathbb{R}}}.
\end{align}
\end{lemma}

\begin{proof}
 We apply $\< \mathbf{J} \bullet \ , D_{\bm z}\boldsymbol{\Phi}\pare{\bm z}\Theta \>$ to equation \cref{eq:modeq,eq:modeq}. By $ \< \mathbf{J} \dot {\bm \eta} \ , D_{\bm z}\boldsymbol{\Phi}\pare{\bm z}\Theta \> = 0 $, due to $ {\bm \eta}\in \boldsymbol{\mathcal{H}}^1_c $, and $  D_{\bm z}\boldsymbol{\Phi}\pare{\bm z} $ independent of $ \bm z $, we get
\begin{multline*}
  \< \mathbf{J} D_{{\bm z}}\boldsymbol{\Phi}\pare{\bm z}(\dot{{\bm z}}-\tilde{{\bm z}}) , D_{\bm z}\boldsymbol{\Phi}\pare{\bm z}\Theta \>
  =- \cancel{\<   \mathbf{J} \boldsymbol{\eta} ,  \mathbf{J}\mathbf{L}_0 D_{\bm z}\boldsymbol{\Phi}\pare{\bm z}\Theta \> }-\cancel{\< \mathbf{J} {\bm R}\pare{\bm z} , D_{\bm z}\boldsymbol{\Phi}\pare{\bm z}\Theta \> } \\
  %---------------------------------------------------
   +\< \mathbf{J}\pare{  f'(\boldsymbol{\Phi}\pare{\bm z}_1) - f'(Q) } \eta _1 \mathbf{j}  \ , D_{\bm z}\boldsymbol{\Phi}\pare{\bm z}\Theta \> %\\
    +\< \mathbf{ J}\pare{ f\pare{ \boldsymbol{\Phi}\pare{\bm z}_1 +\eta _1 }-f(\boldsymbol{\Phi}\pare{\bm z}_1)- f'(\boldsymbol{\Phi}\pare{\bm z}_1)\eta _1  } \mathbf{j}  \ , D_{\bm z}\boldsymbol{\Phi}\pare{\bm z}\Theta \> ,
\end{multline*}
where the second cancelation  follows from \eqref{R:orth} and the first, by \eqref{eq:rp}, from (cf, \cref{eq:rp--1})
\begin{align*}
  \mathbf{J}\mathbf{L}_0 D_{\bm z}\boldsymbol{\Phi}\pare{\bm z}\Theta = D_{\bm z}\boldsymbol{\Phi}\pare{\bm z} \  \tilde{{\bm z}}_0 \bra{ \Theta}  ,
\end{align*}
 and the fact that $ \bm{\eta}\in\cH^1_c $.
Using  Lemma \ref{lem:estF}, we obtain
\begin{align*}
\av{\dot{{\bm z}}-\tilde{{\bm z}}}  \lesssim \norm{   \sech \pare{ \kappa x}    \pare{  f'(\boldsymbol{\Phi}\pare{\bm z}_1) - f'(Q) } \eta _1}_{L^2} \\
%------------------------------------------------------------
+  \norm{   \sech \pare{  \kappa x}    \pare{ f\pare{ \boldsymbol{\Phi}\pare{\bm z}_1 +\eta _1 }-f(\boldsymbol{\Phi}\pare{\bm z}_1)- f'(\boldsymbol{\Phi}\pare{\bm z}_1)\eta _1  } }_{L^2}  \lesssim  \delta ^{p-1} \norm{       \sech \pare{   \kappa x} \eta_1 }_{L^2} .
\end{align*}
\end{proof}
The proof   of \Cref{prop:modp} follows from \Cref{lem:modbound} setting $ \kappa \defeq \xfrac{k_2}{2} $ and  integrating-in-time.

\section{Proof of  Proposition \ref{prop:FGR}: the Fermi Golden Rule} \label{sec:FGR}

Here the Fermi Golden Rule does not refer any more to the condition in \eqref{ass:FGR1}, but rather to the proof  of dissipation of $z_2$ due, ultimately,  to the nonlinear interaction of $z_2$ with the continuous modes of ${\bm u}$.  There is an old history, starting from \cite{Sigal1993,BP1995},
with many significant contributions, like \cite{SW1999,MR2006}. Here we use the argument  of  Kowalczyk and   Martel \cite{KM22}. \\

%{\color{blue} Let us give at first an heuristic idea on the procedure we adopt in the present section. Let us recall that $ \bm{\mathsf{e}}_{\pm} $ defined in \cref{eq:geneigenfJLQ} are the generalized eigenfunctions of the operator $ \mathbf{JL}_0 $, noticing that $ \pare{ \mathbf{JL}_0 }^\ast = -\mathbf{L}_0 \mathbf{J} $ we can deduce that\begin{equation*}\pare{ \mathbf{JL}_0 }^\ast \mathbf{J} \bm{\mathsf{e}}_{\pm}\pare{\xi} = \mp \ii \angles{\xi} \ \mathbf{J} \bm{\mathsf{e}}_{\pm}\pare{\xi} ,\end{equation*}i.e. $ \pare{\mathbf{J} \bm{\mathsf{e}}_{\pm}\pare{\xi} }_{\xi\neq 0} $ are generalized eigenfunctions for the operator $ \pare{ \mathbf{JL}_0 }^\ast $. As such we can define the distorted Fourier transform\begin{equation*}\interior{ \eta }_{\pm} \pare{\xi} \defeq \psc{\eta}{\mathbf{J}  \bm{\mathsf{e}}_{\mp}\pare{\xi} }.\end{equation*}By construction $ \interior{ \eta }_{\pm} \pare{\xi} $ oscillates with frequency $ \pm \angles{\xi} $, hence since $ \dot{z}_2\approx i\lambda \ z_2 $ the quantity\begin{equation*}z_2^2 \ \interior{ \eta }_{\pm} \pare{\xi}\end{equation*}oscillates, approximately, with a frequency $ 2\lambda \pm \angles{\xi} $. Namely we have\begin{equation*}\ddt\pare{z_2^2 \interior{\eta}_{\pm}\pare{\xi}} = \ii \pare{2\lambda \pm \angles{\xi}} \ z_2^2 \interior{\eta}\pare{\xi} + \text{lower order terms}.\end{equation*}It is thus clear that selecting $ \xi=\xi_\lambda \defeq\sqrt{\pare{2\lambda}^2 -1} $ then the quantity $ z_2^2 \  \interior{\eta}_{-}\pare{\xi_\lambda} $ does not present first-order oscillations. \\}

To prove Proposition \ref{prop:FGR},    for the    $\mathbf{g}_{2}$ in Assumption \ref{rem:FGR},    we consider
\begin{align}\label{eq:FGRfunctional}
\mathcal{J}_{\mathrm{FGR}}\defeq\< \mathbf{J} \boldsymbol{\eta},\chi_A \pare{  {z}^{2}_2\mathbf{g}_{2} + \overline{{z}}^{2}_2\overline{\mathbf{g}}_{2}}      \> .
\end{align}

%Computing the time derivative of $\mathcal{J}_{\mathrm{FGR}}$, we have the following estimate.
\begin{lemma}\label{lem:FGR1}
We have
\begin{equation}
	|z_2|^4 \lesssim - \dot{\mathcal{J}}_{\mathrm{FGR}}+ \frac{\dd}{\dd t} \Re  \<  {2^{-1}\lambda  ^{-2}} z_2^4\mathbf{J}\mathbf{G}_2 , \overline{\mathbf{g}}_{2} \>+\delta^2|z_1|^2 +
	A^{-1 } \pare{ \|\boldsymbol{\eta}\|_{\boldsymbol{ \Sigma }_A}^2 +   \| \sech  ( k_2 x) \boldsymbol{\eta}\|_{L^2}^2 }
 \label{eq:lem:FGR11}
\end{equation}
\end{lemma}

\begin{proof}
Differentiating $\mathcal{J}_{\mathrm{FGR}}$  and using \cref{eq:modeq}, we have
\begin{align*}%\label{KMFGR1}
\dot{\mathcal{J}}_{\mathrm{FGR}}=&
\< \mathbf{J} \dot{\boldsymbol{\eta}}, \chi_A \pare{  {z}^{2}_2\mathbf{g}_{2} + \overline{{z}}^{2}_2\overline{\mathbf{g}}_{2}}\>
+\< \mathbf{J}  \boldsymbol{\eta},\chi_A     D_{\bm z} \pare{  {z}^{2}_2\mathbf{g}_{2} + \overline{{z}}^{2}_2\overline{\mathbf{g}}_{2}} \widetilde{{\bm z} }_0        \> +\< \mathbf{J}  \boldsymbol{\eta},\chi_A     D_{\bm z} \pare{  {z}^{2}_2\mathbf{g}_{2} + \overline{{z}}^{2}_2\overline{\mathbf{g}}_{2}} \widetilde{{\bm z} } _R       \>
\\&+\< \mathbf{J}   \boldsymbol{\eta},\chi_A D_{\bm z} \pare{  {z}^{2}_2\mathbf{g}_{2} + \overline{{z}}^{2}_2\overline{\mathbf{g}}_{2}} \pare{\dot{{\bm z}}-\tilde{{\bm z}}}  \> =:A_1+A_2+A_3+A_4.\nonumber
\end{align*}
By  \cref{lem:modbound},  \eqref{eq:relABg} and the following elementary inequality,
\begin{equation}
\label {eq:lem:rhoequiv}
 \| \sech \pare{  \kappa x}  \eta_1\|_{L^2}    \le     \norm{ \sech   \pare{   \frac{2}{A}x}  \eta_1}_{L^2} \le A \norm{\bm{\eta}}_{{\bm \Sigma}_A},
\end{equation}
  we have
\begin{align*}
|A_4|&\lesssim
 \|\boldsymbol{\eta}\chi_A\|_{L^1}\delta ^2 \av{\dot{{\bm z}}-\tilde{{\bm z}} } _{\C^2}
\lesssim
 \delta  ^{p+1}  A ^{\frac{1}{2}}  \norm{ \sech \pare{ \frac{2}{A}x}\boldsymbol{\eta} }_{L^2}  \norm{ \sech \pare{  k_2 x} {\eta}_1 }_{L^2}
\\&\lesssim
  \delta ^{p+1}   A ^{\frac{5}{2}}    \|    \boldsymbol{\eta}\|_{\boldsymbol{ \Sigma }_A}^2 \lesssim A^{-1 }\|    \boldsymbol{\eta}\|_{\boldsymbol{ \Sigma }_A}^2.
\end{align*}
We remark that
\begin{align}\label{eq:fgr11}
  A_2= \< \mathbf{J}  \boldsymbol{\eta},\chi_A     \pare{  2\im \lambda  {z}^{2}_2\mathbf{g}_{2} -  2\im \lambda \overline{{z}}^{2}_2\overline{\mathbf{g}}_{2} } \> .
\end{align}
By   \cref{eq:modeq} and with the cancelation in the 1st line   due to \eqref{eq:fgr11}  and $\mathbf{J} \mathbf{L}_1\mathbf{g}_{2}  =2\im \lambda \mathbf{g}_{2} $,
\begin{align*}&
A_1+A_2= -\< \mathbf{J}  D_{{\bm z}}\boldsymbol{\Phi}\pare{\bm z}(\dot{{\bm z}}-\tilde{{\bm z}}),\chi_A\pare{  {z}^{2}_2\mathbf{g}_{2} + \overline{{z}}^{2}_2\overline{\mathbf{g}}_{2}} \>%\label{KMFGR2}
-\xcancel{
\< \mathbf{J}
	\boldsymbol{\eta},\chi_A  \mathbf{J} \mathbf{L}_0 \pare{  {z}^{2}_2\mathbf{g}_{2} + \overline{{z}}^{2}_2\overline{\mathbf{g}}_{2}}\> }  + \xcancel{A_2}\\
	& -\<
	\boldsymbol{\eta} , \comm{ \mathbf{L}_0}{\chi_A}     \pare{   {z}^{2}_2\mathbf{g}_{2} + \overline{{z}}^{2}_2\overline{\mathbf{g}}_{2} } \>  +\< \mathbf{J}\bm{F}_Q\pare{ \bmz, \bm{\eta}} , \chi_A \pare{   {z}^{2}_2\mathbf{g}_{2} + \overline{{z}}^{2}_2\overline{\mathbf{g}}_{2} } \>
\\&
-\<  \mathbf{J} {\bm R}\pare{\bm z},\chi_A \pare{   {z}^{2}_2\mathbf{g}_{2} + \overline{{z}}^{2}_2\overline{\mathbf{g}}_{2} }  \>
=:A_{11}+A_{12}+A_{13}+A_{14} .
\end{align*}
By  \Cref{lem:modbound}, \eqref{eq:relABg} and \eqref{eq:lem:rhoequiv}, we have
\begin{align*}
|A_{11}|&\lesssim
|\dot{{\bm z}}-\tilde{{\bm z}} |  |z_2 |^2 \lesssim
\delta  ^{p-1} \pare{|{\bm z} |^4+ \norm{\sech \pare{ \xfrac{k _2 x}{2}}\boldsymbol{\eta} }_{L^2}^2  } \le A^{-1 }\pare{ |{\bm z} |^4+\|\boldsymbol{\eta}\|_{\boldsymbol{ \Sigma }_A}^2}.
\end{align*}
By $ \bra{\mathbf{L}_0,\chi_A}  =\begin{pmatrix}
-\chi_A''-2\chi_A'\partial_x & 0 \\ 0 & 0
\end{pmatrix}$,  for a fixed small $\delta _2>0$     we have
\begin{align*}
|A_{12}|
&\lesssim  |z_2 |^2  \pare{\norm{\chi_A'' \eta_1}_{L^1}+\norm{\chi_A' \eta_1'}_{L^1}}\\
&\lesssim |z_2 |^2 \pare{ A^{-3/2}\norm{\sech \pare{ \frac{2}{A} x} \eta_1}_{L^2}+A^{-1/2}\norm{\sech \pare{ \frac{2}{A} x} \eta_1'}_{L^2} }\\
&
\le   \delta _2 |z_2|^4  + C(\delta _2)  A^{-1 }\pare{ \|\sech \pare{ \frac{2}{A} x}\eta_1'\|_{L^2}^2+A^{-2}\|\sech \pare{ \frac{2}{A} x} \eta_1\|_{L^2}^2} ,
\end{align*}
so that we can absorb the $ \delta _2 |z_2|^4$ term in the left hand side of    \eqref{eq:lem:FGR11}.

\noindent By   \Cref{lem:estF}, \eqref{eq:lem:estF2} and  \eqref{eq:lem:estF4},   we have
\begin{align*}
|A_{13}|&\lesssim   |z_2 |^2   \norm{  \chi_A \bm{F}_Q\pare{\bm{\eta}, \bmz} }_{L^1}  \ \lesssim {  A \delta ^{p-1}} \ \pare{ \av{\bmz}^4 +  \norm{\bm{\eta}}_{\bm{\Sigma}_A}^2 }.
\end{align*}
The key term is the following,
\begin{align*}
A_{14}= & \  -2^{-1}\<  \mathbf{J}  D^2_{\bm z} {\bm R}[0] {\bm z}^2,\chi_A\pare{  {z}^{2}_2\mathbf{g}_{2} + \overline{{z}}^{2}_2\overline{\mathbf{g}}_{2}} \> \\
%-----------------------------------------------------
& +\<  \mathbf{J}  \pare{ \frac{1}{2}  D^2_{\bm z} {\bm R}[0]    \pare{\bmz, \bmz} - {\bm R}\pare{\bm z}   },\chi_A\pare{  {z}^{2}_2\mathbf{g}_{2} + \overline{{z}}^{2}_2\overline{\mathbf{g}}_{2}} \>
=:    A_{141} + A_{142}
\end{align*}
We bound the term $ A_{142} $.
 We have
\begin{equation*}
\begin{aligned}
  D^2_{\bm z} {\bm R}[0]   \pare{\bmz, \bmz} = & \  \left. D^2_{\bm z} \left [  D_{\bm z}\bm{\Phi}\pare{\bm z}\ \tilde{{\bm z}}_R \pare{{\bm z}}
 -\left(  f(\boldsymbol{\Phi}\pare{\bm z}_1)   - f(Q)-f' (Q)  \tilde{\boldsymbol{\Phi}}[{\bm z}]_1     \right)  \mathbf{j}   \right ] \right | _{{\bm z}=0} \pare{\bmz, \bmz}
 \\
 %----------------------------------------------------------------
 = & \   D_{\bm z}\bm{\Phi}\pare{\bm z}D^2_{\bm z}\tilde{{\bm z}}_R \bra{0} \pare{\bmz, \bmz} - f''\pare{Q}\mathbf{j} \  \tilde{\bmPhi}\bra{\bmz}^2 ,
\end{aligned}
\end{equation*}
so that
\begin{align}
\label{eq:Taylor_second_order_Remainder}
\frac{1}{2}  D^2_{\bm z} {\bm R}[0]    \pare{\bmz, \bmz} - {\bm R}\pare{\bm z}
&=
D_{\bmz}\bmPhi\pare{\bmz} \pare{ \frac{1}{2} D^2_\bmz \tilde{{\bm z}}_R \bra{0}\pare{\bmz, \bmz}  - \tilde{\bmz}_R }
\\&
%---------------------------------------------
-\left(     f(Q) + f' (Q)  \tilde{\boldsymbol{\Phi}}[{\bm z}]_1   + \frac{f''\pare{Q}}{2} \tilde{\bmPhi}\bra{z}^2 - f(\boldsymbol{\Phi}\pare{\bm z}_1)  \right)  \mathbf{j} . \nonumber
\end{align}
In $\supp \chi (\cdot /A) \subseteq [-2A,2A]$,  we have $\left |   \boldsymbol{\Phi} ({\bm z})_1\right |\ll Q$,  so that by \eqref{eq:Taylor_second_order_Remainder}, $ \av{D_{\bmz}\bmPhi\pare{\bmz}\pare{x} \cdot w}\lesssim e^{-k_2\av{x}}\av{w} $  and  \eqref{eq:142bis},   we have
\begin{equation*}
 \begin{aligned}
  |A_{142}|   \lesssim & \  |{\bm z}| ^2  \int _{-2A}^{2A} \set{ \av{ \frac{1}{2} D^2_\bmz \tilde{{\bm z}}_R \bra{0}\pare{\bmz, \bmz}  - \tilde{\bmz}_R }
  e^{-k_2|x|}  + \left|     f(Q) + f' (Q)  \tilde{\boldsymbol{\Phi}}[{\bm z}]_1   + \frac{f''\pare{Q}}{2} \tilde{\bmPhi}\bra{z}^2 - f(\boldsymbol{\Phi}\pare{\bm z}_1)  \right|} \dd x
   \\
%--------------------------------------------------------
    \lesssim  & \    o \pare{ \av{{\bm z}} ^4 }  + |{\bm z}| ^5 \norm{  f'''(Q) e^{-3k_2 |x|} }_{L^1(\R )} ,
 \end{aligned}
 \end{equation*}
  so that, by   $\av{ f'''(Q\pare{x}) e^{-3k_2 |x|}}\sim e^{-k_2\av{x}} $, we obtain
 \begin{equation}
 \label{eq:142}
  |A_{142}|   \lesssim   o \pare{ \av{{\bm z}} ^4 }.
 \end{equation}
 Next we  write
\begin{align*}
  A_{141} &= -2^{-1}\<  \mathbf{J}\pare{ z_2^2\mathbf{G}_2  + \overline{z}_2^2\overline{\mathbf{G}}_2  } ,   {z}^{2}_2\mathbf{g}_{2} + \overline{{z}}^{2}_2\overline{\mathbf{g}}_{2} \> \\& -2^{-1}\<  \mathbf{J}\pare{ z_2^2\mathbf{G}_2  + \overline{z}_2^2\overline{\mathbf{G}}_2   } ,(1-\chi_A)\pare{  {z}^{2}_2\mathbf{g}_{2} + \overline{{z}}^{2}_2\overline{\mathbf{g}}_{2}} \> +O\pare{  z_1 ^4   } +  \delta _2    O\pare{  z_2 ^4  }
\end{align*}
for a   $\delta _2 >0$ arbitrarily small but fixed.  The first term in the second line, can be absorbed in the third.  So,
\begin{align*}
  A_{142} &= - |z_2| ^4 \<  \mathbf{J}\mathbf{G}_2 , \mathbf{g}_{2} \> - 2\Re \<  z _2 ^4\mathbf{J}\mathbf{G}_2 , \overline{\mathbf{g}}_{2} \> +O\pare{  z_1 ^4   } +  \delta _2    O\pare{  z_2 ^4  }.
\end{align*}
By elementary computation,    for $\tilde{{\bm z}} =:((\tilde{{\bm z}})_1,(\tilde{{\bm z}})_2) ^ \intercal $ and using \eqref{eq:rp--1}  and \eqref{eq:tildez},    we have
\begin{align*}
   \frac{d}{\dd t}z_2 ^4&= 4\im \lambda (\tilde{{\bm z}})_2 z_2 ^3 +4\im \lambda \pare{ \dot z_2- (\tilde{{\bm z}})_2 } z_2 ^3\\&  =  -4\lambda ^2z_2 ^4+4\im \lambda (\tilde{{\bm z}}_{R})_2 z_2 ^3 +4\im \lambda \pare{ \dot z_2- (\tilde{{\bm z}})_2 } z_2 ^3.
\end{align*}
Hence we conclude the following, which with \eqref{lem:modbound1} yields the statement,
\begin{align*}
  A_{142} &= - |z_2| ^4 \<  \mathbf{J}\mathbf{G}_2 , \mathbf{g}_{2} \>  + \frac{d}{\dd t}\Re \<  2^{-1}\lambda ^{-2} z _2 ^4\mathbf{J}\mathbf{G}_2 , \overline{\mathbf{g}}_{2} \>
   +O\pare{  z_1 ^4   } +  \delta _2    O\pare{  z_2 ^4  }  + O\pare{  |\dot {{\bm z}} -\tilde{{\bm z}} |^2} .
\end{align*}
\end{proof}

 \begin{lemma}\label{lem:esz1L4} We have
 \begin{equation}\label{eq:esz1L41}
 	 |z_1| ^2  \lesssim   \frac{d}{\dd t} \Re z_1 ^2  +    |z_2| ^4 +   \delta^{2(p-1)}     \| \sech  ( k_2 x) \boldsymbol{\eta}\|_{L^2}^2   .
 \end{equation}
 \end{lemma}
 \begin{proof} Using \eqref{eq:rp--1} and \eqref{eq:tildez}, we have
    \begin{align*}
       \frac{d}{\dd t}z_1^2 = 2\nu _0  |z_1| ^2 + 2z_1 (\dot z_1 - \nu _0 \overline{z}_1) =       2\nu _0  |z_1| ^2 +      2z_1 (\dot z_1 - ( \tilde{{\bm z}} )_1) +  2z_1 ( \tilde{{\bm z}}_R )_1.
    \end{align*}
Taking the real part of the above equation and using   Young inequality, for any $\theta>0$   we get
\begin{align*}
2\nu _0 |z_1| ^2\leq &\frac{d}{\dd t} \Re z_1 ^2+ 2\theta |z_1|^2 + \frac{1}{\theta} | \dot{{\bm z}}-\tilde{{\bm z}}|^2+ \frac{1}{\theta}| \tilde{{\bm z}}_R|^2  \\
\stackrel{\eqref{item:rem_better_1}}{\leq} & \frac{d}{\dd t} \Re z_1 ^2+ 2\theta |z_1|^2 + \frac{1}{\theta} | \dot{{\bm z}}-\tilde{{\bm z}}|^2+ \frac{2 C_1^2}{\theta}  |{\bm z_1}|^4+ \frac{2 C_1^2}{\theta}  |{\bm z_2}|^4\\
{\leq} & \frac{d}{\dd t} \Re z_1 ^2+ (2\theta+\frac{2 C_1^2}{\theta}  \delta^2) |z_1|^2 + \frac{1}{\theta} | \dot{{\bm z}}-\tilde{{\bm z}}|^2+ \frac{2 C_1^2}{\theta}  |{\bm z_2}|^4
\end{align*}
  Choosing $ \theta\defeq \frac{\nu_0}{4}$ and using $ \delta\ll 1$, we get the following, which proves \eqref{eq:esz1L41},
    \begin{equation*}
 |z_1| ^2\lesssim \frac{1}{\nu_0^2} | \dot{{\bm z}}-\tilde{{\bm z}}|^2+\frac{1}{\nu_0} \frac{d}{\dd t} \Re z_1 ^2+\frac{1}{\nu_0^2}|{\bm z_2}|^4\lesssim \frac{\delta ^{2(p-1)}}{\nu_0^2}    \norm{     \sech \pare{  \kappa x} \boldsymbol{\eta}\pare{t} } _{L^2\pare{\mathbb{R}}}^2+\frac{1}{\nu_0} \frac{d}{\dd t} \Re z_1 ^2+\frac{1}{\nu_0^2}|{\bm z_2}|^4.
    \end{equation*}

 \end{proof}

\begin{proof}[Proof of Proposition \ref{prop:FGR}]
	Integrating \eqref{eq:lem:FGR11}  for $t \in I$ and by \eqref{eq:main11},
\begin{align}\nonumber
  \| z_2  \|  ^{4}_{L^4(0,t)} &\lesssim  \left |  \left .\pare{ {\mathcal{J}}_{\mathrm{FGR}}-  \Re  \<  {2^{-1}\lambda  ^{-2}} z _2 ^4\mathbf{J}\mathbf{G}_2 , \overline{\mathbf{g}}_{2} \> }  \right ] _{0} ^{t}\right |  +\delta ^2 \| z_1  \|  ^{2}_{L^2(0,t)} +
A^{-1/2}    \|   \boldsymbol{\eta}  \|_{L^2((0,t), \boldsymbol{ \Sigma }_A  \cap   \boldsymbol{L }^2_{-\kappa})}^2\\& = o _\epsilon  (1)      \epsilon ^2  .\label{eq:esz1L421}
\end{align}
Integrating \eqref{eq:esz1L41}   for $t \in I$ and by \eqref{eq:main11} we obtain the following, which completes the proof,
 \begin{align*}
      \| z_1  \|  ^{2}_{L^2(0,t)}&\lesssim  \left |  \left . \Re z_1 ^2 \right ] _{0} ^{t}\right |  + \delta _2 ^{-1}\| z_2  \|  ^{4}_{L^4(0,t)}  +
A^{-1/2}    \| \sech  ( k_2 x) \boldsymbol{\eta}\|_{L^2 ((0,t),L^2)}^2= o _\epsilon  (1)      \epsilon ^2.
\end{align*}

\end{proof}

\section{Proof of Proposition \ref{prop:1stvirial}: the first virial estimate.}
\label{sec:1virial}

\begin{notation}
For the $\chi$ in \Cref{notation:chi} and for $C>0$, we set
\begin{align}\label{def:zetaphi}
\zeta_C(x)\defeq\exp\left(-\frac{|x|}{C}(1-\chi(x))\right), && \varphi_C(x)\defeq\int_0^x \zeta_C^2(y)\,\dd y\  &&  S_C\defeq\frac{1}{2}\varphi_C'+\varphi_C\partial_x.
\end{align}
\end{notation}
We define functionals (notice that $S_A\mathbf{J}$ and   $\sigma_3\mathbf{J}$ are selfadjoint)
\begin{align*}
&\mathcal{I}_{\mathrm{1st},1}\defeq\frac{1}{2}\< \mathbf{J} \boldsymbol{\eta},S_A\boldsymbol{\eta}\>, \\
& \mathcal{I}_{\mathrm{1st},2}\defeq\frac{1}{2}\< \mathbf{J}  \boldsymbol{\eta},\sigma_3\zeta_A^4 \boldsymbol{\eta}\> && \text{where} && \sigma_3\defeq \begin{pmatrix}
	1 & 0 \\ 0 & -1
\end{pmatrix}.
\end{align*}

\begin{lemma}\label{lem:1stV1}
There exist $\delta_0>0$, $A_0>0$ and $ 0 < \kappa_0 \ll k_2 $ s.t.  if $\delta\in\pare{0,  \delta_0 }$, for any $ \kappa\in\pare{0, \kappa_0} $ and $A>A_0$, we have
\begin{equation}
\norm{    \sech  \left(   \frac{2}{A}x\right)\eta_1'}_{L^2}^2 + A^{-2}\norm{    \sech  \left(   \frac{2}{A}x\right) \eta_1}_{L^2}^2   \lesssim \dot{\mathcal{I}}_{\mathrm{1st},1}+ \|   \sech  \left(  \kappa x\right) \boldsymbol{\eta}\|_{L^2}^2+    |{\bm z} |^4. \label{eq:lem:1stV1}
\end{equation}
\end{lemma}

\begin{proof}
Using   \eqref{nuovaequi} and the fact that $ S_A$ is skew-adjoint ,  we have
\begin{equation}
\label{eq:1V1}
\begin{aligned}
\dot{\mathcal{I}}_{\mathrm{1st},1}&=\< \mathbf{J} \dot {\boldsymbol{\eta}},S_A\boldsymbol{\eta}\> =    \< \mathbf{J}S_A  D\boldsymbol{\Phi}\pare{\bm z}(\dot{{\bm z}}-\tilde{{\bm z}}),\boldsymbol{\eta}\> +\<\partial_x^2 \eta_1  ,S_A\eta_1 \>
\\
%----------------------------------------
& \ + \<\left( f\left( \boldsymbol{\Phi}\pare{\bm z}_1 +\eta _1 \right)-f(\boldsymbol{\Phi}\pare{\bm z}_1) \right),S_A\eta_1\>  -\< \mathbf{J}  {\bm R}\pare{\bm z},S_A\boldsymbol{\eta}\>
 \eqdef   \
 B_1+B_2+B_3+B_4 .
\end{aligned}
\end{equation}
By Lemma 1 in Kowalczyk et al. \cite{KMM2022},
\begin{align}\label{B2}
& B_2 =     -\|\zeta_A\eta_1'\|_{L^2}^2   - \frac{1}{2} \int _\R \(   \zeta   _A \zeta '' _A - \(\zeta '  _A\) ^2    \) \eta_1 ^2 dx     \text{ with }
\\& \label{cbound}
\left |  \zeta   _A \zeta '' _A - \(\zeta '  _A\) ^2    \right | \lesssim \frac{1_{1\le |x|\le 2}}{A}.
\end{align}
Using  Lemma \ref{lem:modbound} to bound $| \dot {{\bm z}} -\tilde{{\bm z}}|$, we have, for $\kappa<k_2$,
\begin{equation}\label{B1}
|B_{1}| \le  | \dot {{\bm z}} -\tilde{{\bm z}}|  \| \sech(\kappa x)  \boldsymbol{\eta}\| _{L^2} \lesssim \delta ^{p-1}\| \sech(\kappa x)  \boldsymbol{\eta}\| _{L^2}^2.
\end{equation}
 We   decompose
 \begin{equation}\label{B4}
 	B_4=   \< \mathbf{J}S_A D_{\bm z}\boldsymbol{\Phi}\pare{\bm z}\tilde{{\bm z}}_R ,\boldsymbol{\eta}\>+ \<      f(\boldsymbol{\Phi}\pare{\bm z}_1)   - f(Q)-f' (Q)  \tilde{\boldsymbol{\Phi}}[{\bm z}]_1      ,S_A {\eta}_1\> =: B_{41}+ B_{42}.
 \end{equation}
By \Cref{lem:rpcorr}       we have
\begin{align}
 | B_{41}| = |\< \mathbf{J}S_A D_{\bm z}\boldsymbol{\Phi}\pare{\bm z}\tilde{{\bm z}}_R ,\boldsymbol{\eta}\> |
  \lesssim  & |\tilde{{\bm z}}_R| \  \| \sech(\kappa x)\boldsymbol{\eta} \| _{L^2 } \notag \\
   \lesssim    |{\bm z}|^2 \|\sech(\kappa x) \boldsymbol{\eta} \| _{L^2}
    \le &  |{\bm z}|^4+ \|\sech(\kappa x) \boldsymbol{\eta} \| _{L^2}^2.\label{bound_B41}
 \end{align}
By \eqref{eq:Ijest2}  and for $\delta_2>0$ small fixed,
\begin{align}
 | B_{42}| \lesssim & |{\bm z}|^2 \|  e^{-|x| }\varphi _A \cosh \pare{ \frac{2}{A}x} \| _{L^2} \| \sech \pare{ \frac{2}{A}x} \eta _1 '\| _{L^2}  + |{\bm z}|^2\|\sech(\kappa x)  {\eta}_1 \| _{L^2}
 \nonumber \\ \le &   \frac{2}{\delta_2} |{\bm z}|^4+ \delta_2\|\sech(\kappa x)  {\eta} _1\| _{L^2}^2 +\delta_2 \| \sech \pare{ \frac{2}{A}x} \eta _1 '\| _{L^2}^2     .\label{bound_B42}
 \end{align}
We decompose
 \begin{equation}\label{B3}
   B_3=   \< f\left( \boldsymbol{\Phi}\pare{\bm z}_1 +\eta _1 \right)-f(\boldsymbol{\Phi}\pare{\bm z}_1)- f\left(  \eta _1 \right) ,S_A {\eta}_1\> + \<  f\left(  \eta _1 \right) ,S_A {\eta}_1\>=:B_{31}+B_{32}.
 \end{equation}
For the pure non-linear term $B_{32}$, we have
\begin{align}
B_{32}=& \frac{1}{p+1}\int_{\R}  \partial_x\big(|\eta_1|^{p+1}\big) \varphi_A(x)\, {\rm d}x\, + \frac12 \int_{\R} \zeta_A^2(x) |\eta_1|^{p+1}\,{\rm d} x\notag \\
=&\frac{p-1}{2(p+1)} \int_{\R} \zeta_A^2(x) |\eta_1|^{p+1}\,{\rm d} x.
\end{align}
By     a  clever but very simple integration by parts      Kowalczyk et al. \cite{KMM2022}, see inequality (36)  therein,   get the following  bound, which  is a striking validation of their   dispersion theory,
\begin{equation}\label{bound_B32}
	|B_{32}|\lesssim \delta^{\frac{2p}{3}} \| (\zeta_A \eta_1)'\|_{L^2}^2.
\end{equation}
 Next, since for the concave function $t\to |t| ^{p-1}$ for $1<p\le 2$ we have
$$
\big||y+h|^{p-1}-|y|^{p-1}\big|\leq |h|^{p-1} \quad \text{for any} \quad y,h\in \C,
$$
we conclude
\begin{align}
|f\left( \boldsymbol{\Phi}\pare{\bm z}_1 +\eta _1 \right)-f(\boldsymbol{\Phi}\pare{\bm z}_1)- f\left(  \eta _1 \right)|\leq &\int_0^1 \big|f'(\boldsymbol{\Phi}\pare{\bm z}_1+ \tau \eta _1)- f'(\tau \eta_1)\big|{\rm d} \tau\, \cdot  | \eta_1|\notag  \\
\leq & p \av{\boldsymbol{\Phi}\pare{\bm z}_1}^{p-1} |\eta_1|\lesssim \sech( k_2(p-1) x ) | \eta_1|.\label{bonina}
\end{align}
Then, for $ \tfrac{2}{A}<\kappa< k_2 (p-1)/2$, we have
\begin{align}
  |B_{31}| \lesssim \int_{\R} \sech^2(\kappa x ) | \eta_1| | \eta_1'| +\int_{\R} \sech^2( \kappa x ) | \eta_1|^2\leq (1+\tfrac{1}{\delta_2})\|\sech(\kappa x ) \eta_1\|_{L^2}^2+ \delta_2\| \sech\left(\frac{2}{A}x\right) \eta_1'\|_{L^2}^2. \label{B31}
\end{align}
Collecting all the terms in \cref{eq:1V1,B2,B4,B3}  and their bounds in \cref{B1,bound_B41,bound_B42,bound_B32,B31}, we get
\begin{align*}
&\|(\zeta_A \eta_1)'\|_{L^2}^2\lesssim   -\dot{\mathcal{I}}_{\mathrm{1st},1}
 +\|\sech  \left(  \kappa x\right)\boldsymbol{\eta}\|_{L^2}^2+  |{\bm z} |^4+\delta_2\| \sech\big(\frac{2}{A}x\big) \eta_1'\|_{L^2}^2.
\end{align*}
  Finally, we use the following, which is analogous to  (19) of \cite{KM22} and is proved in  \cite{CMS2023}, see formula (6.5),
\begin{equation}\label{eq:KM19}
\| \sech  \left(   \frac{2}{A}x\right) \eta_1'\|_{L^2}^2+ \frac{1}{A^2}\| \sech  \left(   \frac{2}{A}x\right) \eta_1\|_{L^2}^2\lesssim  \| (\zeta_A \eta_1)'\|_{L^2}^2 + \frac{1}{A}\|\sech  \left(  \kappa x\right)\eta_1\|_{L^2}^2.
\end{equation}
This, choosing $\delta_2>0$ small, yields \eqref{eq:lem:1stV1}.
\end{proof}

\begin{lemma}\label{lem:1stV2}
With the same hypothesis of \Cref{lem:1stV1}, we have
\begin{align}
\norm{\sech  \left(   \frac{2}{A}x\right)\eta_2}_{L^2}^2  \lesssim  \dot{\mathcal{I}}_{\mathrm{1st},2}+\norm{\sech  \left(   \frac{2}{A}x\right) \eta_1' }_{L^2}^2+ \norm{ \sech  \left(   \frac{2}{A}x\right)\eta_1}_{L^2}^2+ \norm{\boldsymbol{\eta} }_{L^2_{-\kappa}}^2+ |{\bm z} |^4.\label{eq:lem:1stV2}
\end{align}

\end{lemma}
\begin{proof}
We have, from \cref{nuovaequi}
\begin{align*}
\dot{\mathcal{I}}_{\mathrm{1st},2} & =-\< \mathbf{J} D_{{\bm z}}\boldsymbol{\Phi}\pare{\bm z}(\dot{{\bm z}}-\tilde{{\bm z}}),\sigma_3\zeta_A^4 \boldsymbol{\eta} \>
-\< \begin{pmatrix}
	-\partial_{x}^2+ 1 &0\\0&1
\end{pmatrix} \boldsymbol{\eta}
,
\sigma_3\zeta_A^4 \boldsymbol{\eta}\>
\\
& \qquad +  \<   \left( f\left( \boldsymbol{\Phi}\pare{\bm z}_1 +\eta _1 \right)-f(\boldsymbol{\Phi}\pare{\bm z}_1) \right),
\zeta_A^4  {\eta}_1\> -\< \mathbf{J}  {\bm R}\pare{\bm z},
\sigma_3\zeta_A^4 \boldsymbol{\eta}\>\\
&=:  \  C_1+C_2+C_3+C_4
\end{align*}
For the main term $C_2$, we have
\begin{equation}\label{eq:C2_bound1}
C_2 =- \<(-\partial_{x}^2+ 1) \eta_1, \zeta_A^4\eta_1\> + \norm{\zeta_A^2\eta_2} _{L^2}^2,
\end{equation}
with
\begin{align}
\label{eq:C2_bound2}
|\<(-\partial_{x}^2+ 1)\eta_1,\zeta_A^4\eta_1\>|\lesssim \norm{\sech  \left(   \frac{2}{A}x\right) \eta_1'}_{L^2}^2+\norm{\sech  \left(   \frac{2}{A}x\right) \eta_1}_{L^2}^2.
\end{align}
We consider  the remainder terms, starting with $ C_1 $. By \Cref{lem:modbound} and
\begin{equation}\label{scemetta}
	\av{D_\bmz \bmPhi\pare{\bmz}(w)\zeta_A^4}\lesssim  \sech\left(  \frac{4}{A}x\right)\sech  \left(  \kappa x\right) \av{w}, \quad \forall w \in \C^2,
\end{equation} we obtain
\begin{align}\label{eq:C1_bound}
|C_1|&\lesssim |\dot{{\bm z}}-\tilde{{\bm z}}|\norm{\sech  \left(  \kappa x\right)\boldsymbol{\eta}}_{L^2}\lesssim \delta ^{p-1} \norm{\boldsymbol{\eta}}_{L^2_{-\kappa}}^2 .
\end{align}
We next focus on $C_4$, with
\begin{equation}
\label{eq:C4_decomposition}
 \begin{aligned}
 C_4 & \ = C_{41}+ C_{42} , \\
 C_{41} & \ \defeq -\< \mathbf{J} D_{\bm z}\boldsymbol{\Phi}\pare{\bm z}\tilde{{\bm z}}_R ,  \sigma_3\zeta_A^4 \boldsymbol{\eta} \>
 \\
  C_{42}& \ \defeq  \<{      f(\boldsymbol{\Phi}\pare{\bm z}_1)   - f(Q)-f' (Q)  \tilde{\boldsymbol{\Phi}}[{\bm z}]_1}, {\sigma_3\zeta_A^4 \boldsymbol{\eta}} \>
 \end{aligned}
\end{equation}
Combining  \eqref{scemetta} and \eqref{item:rem_better_1}, we get
\begin{align}\label{eq:C41_bound}
 | C_{41}| \le \av{\bm{z}}^2\norm{\boldsymbol{\eta}}_{L^2_{-\kappa}}\le  |{\bm z}|^4+ \| \boldsymbol{\eta} \| _{L^2_{- \kappa}}^2.
 \end{align}
By \eqref{eq:Ijest2}, we have
\begin{equation}\label{eq:C42_bound}
  | C_{42}|\leq \av{\bm{z}}^2\norm{\boldsymbol{\eta}}_{L^2_{-\kappa}} \le      |{\bm z}| ^4 +   \norm{\boldsymbol{\eta}}^2_{L^2_{-\kappa}} ,
\end{equation}
so that \cref{eq:C4_decomposition,eq:C42_bound,eq:C41_bound} give that
\begin{equation} \label{eq:C4_bound}
|C_4| \lesssim \av{\bmz}^4
+ \norm{   \boldsymbol{\eta}  } ^2 _{L^2_{-\kappa}}.
\end{equation}
To bound $C_3$, we notice that by
$$
 \av{f\left( \boldsymbol{\Phi}\pare{\bm z}_1 +\eta _1 \right)-f(\boldsymbol{\Phi}\pare{\bm z}_1)}\leq p \av{\boldsymbol{\Phi}\pare{\bm z}_1}^{p-1} | \eta_1| + \av{\eta_1}^p
$$
 we obtain \begin{equation}\label{loveroC3}
	|C_3| \lesssim \| \sech\left(\kappa x\right)\eta_1\|_{L^2}^2+ \delta^{p-1}\| \sech\left(\frac{2}{A} x\right)\eta_1\|_{L^2}^2.
\end{equation}
Collecting the estimates in \cref{eq:C2_bound1,eq:C2_bound2,eq:C1_bound,loveroC3}, we have the conclusion.
\end{proof}

\begin{proof}[Proof of Proposition \ref{prop:1stvirial}]
It is immediate from  $|\mathcal{I}_{\mathrm{1st},1}|\lesssim A\delta^2$, $|\mathcal{I}_{\mathrm{1st},2}|\lesssim \delta^2$,   Lemmas \ref{lem:1stV1} and \ref{lem:1stV2} and the definition \eqref{eq:normA} of
$\| \cdot \| _{ \boldsymbol{ \Sigma }_A}.$
\end{proof}

\section{Technical estimates on the Darboux transform}\label{sec:tech}

 For the $N$  in \eqref{def:valueN}   and  the  $ \mathcal{A}$ in \eqref{def:calA}, we consider
\begin{align} \label{def:Tg} &
\mathcal{T}\defeq\<\im \varepsilon \partial_x\>^{- \pare{1+N} }\mathcal{A}^\ast.
\end{align}
We will use the following   formula proved in \cite{CM2022}, with $L^2_c(L_0)$ the continuous spectrum component in $L^2(\R , \C)$  of $L_0$ and $P_c$ the corresponding orthogonal projection,
\begin{align}\label{eq:Tinverse}
  {\bm u}=\prod_{j=0}^{N}R _{L_0}(\mu _j ) P_c \mathcal{A} \<   \im \varepsilon\partial_x\>^{N+1}  \mathcal{T} {\bm u} && \forall \ {\bm u}\in L^2_c(L_0).
\end{align}
In  \cite{CM2022}  the following result  is proved.
\begin{lemma}\label{lem:coer6A}
There exists a $ 0 < \kappa_0 \ll k_2  $ such that for any $ \kappa\in\pare{0, \kappa_0} $ there exists a constant $C_\kappa $ such that  for all  $0<\varepsilon\le  1$ and $\bm{w} \in L ^{2}_{-\frac{\kappa}{2}} $
\begin{align}\label{eq:coer6A1} \norm{
   \prod_{j=0}^{N}R _{L_0}(\mu _j ) P_c \mathcal{A}^* \<   \im \varepsilon\partial_x\>^{N+1}   \bm{w}  } _{L ^{2}_{-\kappa}} \le  C_\kappa    \|
     \bm{w}  \| _{L ^{2}_{-\frac{\kappa}{2}}} .
\end{align}
\end{lemma}
\qed

In \cite[Lemma 7.5]{CMS2023}, the following is proved.

\begin{lemma} \label{lem:KM1}
there exists $ A_0, \varepsilon_0 > 0 $ such that for any $ A > A_0 $ and $ \varepsilon\in\pare{0, \varepsilon_0} $ and for any $u\in H^1$ we have
\begin{align}\label{eq:KM1}
\norm{ \sech \left(   \frac{4}{A}    x \right)    \mathcal{T} u}_{L^2}\lesssim & \
 \varepsilon^{-\pare{1+N}}\norm{\sech \left(    \frac{2}{A}    x\right)u}_{L^2},\\ \label{eq:KM2}
\norm{ \sech \left(    \frac{4}{A}    x  \right) \partial_x\mathcal{T} u }_{L^2}\lesssim & \
 \varepsilon^{-\pare{1+N}}\norm{ \sech \left(  \frac{2}{A}    x\right)u'}_{L^2}+ \varepsilon^{- N }\norm{ u}_{L^2_{-\kappa}}.
\end{align}

\end{lemma}
\qed

%Let $ N=3 $ and
%\begin{align}\label{eqdefVN1}
%  V_{4} = \begin{cases}
%  -k_{3}k_{4 }\frac{2}{p-1}Q ^{p-1} >0 \text{ if    $2>p>5/3$}
%\\  0 \text{ if    $p=2$.}
%\end{cases}
%\end{align}
In \cite[Lemma 7.6]{CMS2023}   the following, obvious for $ V_{N+1} =0$,  is proved.

\begin{lemma} \label{lem:KM2}
For any $u\in H^1$,
\begin{align}  \label{eq:KM3}
\norm{     \comm{\<\im \varepsilon \partial_x\>^{-\pare{1+N}}}{V_{N+1}}\mathcal{A} u }_{L^2}& \ \lesssim
 \varepsilon \norm{ \mathcal{T}u} _{L^2_{-\kappa}}, \\
   \label{eq:KM3b} \norm{  \cosh \left( \frac{\kappa}{2} x\right)   \comm{\<\im \varepsilon \partial_x\>^{-\pare{1+N}}}{V_{N+1}} \mathcal{A} u}_{L^2}& \ \lesssim
 \varepsilon \norm{ \mathcal{T}u } _{L^2_{-\frac{\kappa}{2}}}.
\end{align}
\end{lemma} \qed

\section{Proof of Proposition \ref{prop:2ndvirial}: virial estimates for the transformed equation} \label{proof:2dnv}

Using the $\mathcal{T}$  in  \eqref{def:Tg} with $ N=3 $ for $ p\in\pare{\frac{5}{3}, 2} $ and $ N=2 $ for $ p=2 $, we consider the new   variable
\begin{align}   \label{def:vBg}
{\bm v}\defeq\mathcal{T}\boldsymbol{\eta}.
\end{align}
Here   ${\bm v}$ is even for $5/3<p<2$ but, crucially,  is odd for $p=2$.
For ${\bf L}_{ N+1}$, cfr. \cref{eq:linearNLKG_2},    ${\bm v}$ satisfies
\begin{multline}
\label{eq:vBg}
\dot{{\bm v}}=
-\mathcal{T}D\boldsymbol{\Phi}\pare{\bm z}(\dot{{\bm z}}-\tilde{{\bm z}})
+\mathbf{J}\left({\bf L}_{ N+1}{\bm v}
+\begin{pmatrix}
\comm{\< \ii\varepsilon \partial_x\>^{-\pare{N+1}}}{V_{N+1}} & 0 \\ 0 & 0
\end{pmatrix}
\mathcal{A}^*\boldsymbol{\eta}\right)\\
%--------------------------------------------------------
+        \mathcal{T} \left  [\left(  f'(\boldsymbol{\Phi}\pare{\bm z}_1) - f'(Q) \right) \eta _1 \mathbf{i}  + \left( f\left( \boldsymbol{\Phi}\pare{\bm z}_1 +\eta _1 \right)-f(\boldsymbol{\Phi}\pare{\bm z}_1)- f'(\boldsymbol{\Phi}\pare{\bm z}_1)\eta _1  \right) \mathbf{i}
 -  {\bm R}\pare{\bm z}\right ]   .
\end{multline}
From Lemma \ref{lem:coer6A}, we have
\begin{align}
\|      \sech ( \kappa x)\boldsymbol{\eta}\|_{L^2}\lesssim \norm{ \sech \pare{  \frac{\kappa}{2} \  x }{\bm v}}_{L^2}.\label{eq:key1}
\end{align}
%Therefore, to bound the local norm of $\boldsymbol{\eta}$, it suffices to bound a local norm of ${\bm v}$.
Set, for the $\varphi_B$ defined in \eqref{def:zetaphi},
\begin{align*}
\psi_{A,B}=\chi_A^2 \varphi_B,
&&
  \tilde{S}_{A,B}=\frac{1}{2}\psi_{A,B}'+\psi_{A,B}\partial_x,
\end{align*}
and consider the functionals
\begin{align*}
\mathcal{I}_{\mathrm{2nd},1}\defeq\frac{1}{2}\< \mathbf{J} {\bm v},\tilde{S}_{A,B}{\bm v}\> ,
&&
 \mathcal{I}_{\mathrm{2nd},2}\defeq\frac{1}{2}\< \mathbf{J}  {\bm v},\sigma_3 e^{-\kappa \<x\>}{\bm v}\> .
\end{align*}

\begin{lemma}\label{lem:2v1}
	We have
\begin{align} &
\|v_1'\|_{L^2_{-\frac{\kappa}{2}}}^2+\|v_1\|_{L^2_{-\frac{\kappa}{2}}}^2+\dot{\mathcal{I}}_{\mathrm{2nd},1}
\lesssim   \left(\varepsilon^{-\pare{N+1}}A^2\delta+A^{-1/2}\right)\|\boldsymbol{\eta}\|  _{\boldsymbol{ \Sigma }_A}^2
+  |{\bm z} |^4.
\label{eq:lem:2v1} \end{align}
\end{lemma}

\begin{proof}
By \cref{eq:vBg}, we have
\begin{align*}
\dot{\mathcal{I}}_{\mathrm{2nd},1}= & \  -\< \mathbf{J} \mathcal{T}D\boldsymbol{\Phi}({\bm z})(\dot{{\bm z}}-\tilde{{\bm z}}),\tilde{S}_{A,B}{\bm v}\> -
\<{\bf L}_{ N+1}{\bm v},\tilde{S}_{A,B}{\bm v}\> \\& +\<\begin{pmatrix}
\comm{\< \ii\varepsilon \partial_x\>^{-\pare{N+1}}}{V_4} & 0 \\ 0 & 0
\end{pmatrix}\mathcal{A}^*\boldsymbol{\eta},\tilde{S}_{A,B}{\bm v}\>
  \\&+\<\mathcal{T} \left( \left(  f'(\boldsymbol{\Phi}\pare{\bm z}_1) - f'(Q) \right) \eta _1   + \left( f\left( \boldsymbol{\Phi}\pare{\bm z}_1 +\eta _1 \right)-f(\boldsymbol{\Phi}\pare{\bm z}_1)- f'(\boldsymbol{\Phi}\pare{\bm z}_1)\eta _1  \right)   \right)
 ,\tilde{S}_{A,B} {v}_1\>  \\& -\<\mathcal{T}  {\bm R}\pare{\bm z} ,\tilde{S}_{A,B}{\bm v}\>
=:   \ D_1+D_2+D_3+D_4+D_5.
\end{align*}
 Following \cite[Section 5]{KM22},  for the main term $D_2$  we have
\begin{align}\label{def:xi}
D_2=-\<L_{N+1}v_1,\widetilde{S}_{A,B}v_1\>= -\int \left(  \xi _1 ^{\prime  2}+V_B \xi _1 \right)\,dx+D_{21}  , && \xi _1\defeq \chi_A\zeta_B v_1,
\end{align}
  where
\begin{align}\label{eq:defvb}&
 V_B = \frac{1}{2} B^{-1}\left(\chi'' |x|+2\chi' \frac{x}{|x|}\right)  -\frac{1}{2}\  \frac{\varphi _B}{\zeta _B^2}V'_{N+1} \text{ and}\\&
D_{21}=\frac{1}{4}\int (\chi_A^2)'(\zeta_B^2)'v_1^2+\frac{1}{2}\int \left(3(\chi_A')^2+\chi_A''\chi_A\right)\zeta_B^2v_1^2-\int (\chi_A^2)'\varphi_B(v_1')^2+\frac{1}{4}\int (\chi_A^2)''' \varphi_B v_1^2. \nonumber
\end{align}
We claim that
\begin{align}\label{eq:KMlemma3}
\int ( \xi _1 ^{\prime  2}+V_B \xi _1 )\,dx  +A^{-1}\|\boldsymbol{\eta}\|_{\boldsymbol{ \Sigma }_A}^2\gtrsim \left(\norm{\sech \left( \frac{\kappa}{2} x\right) v_1'}_{L^2}^2+\norm{\sech \left( \frac{\kappa}{2} x\right) v_1}^2\right).
\end{align}
The proof is like in \cite[Lemma 3]{KM22}.  We have, by $0<\zeta  _B\le 1$,
\begin{align*} & \int _{|x|\le A}\sech \left(  \kappa x\right)v_1^2\le  \int _{|x|\le A}\sech \left(  \frac{\kappa}{2} x\right) \zeta ^2_B   v_1^2\le  \int _{|x|\le A}\sech \left(  \frac{\kappa}{2} x\right) \xi_1^2.
\end{align*}
We have
\begin{align*} & \int _{|x|\le A}\sech \left(  \kappa x\right)v_1 ^{\prime 2}\le  \int _{|x|\le A}\sech \left(  \frac{\kappa}{2} x\right) \left( \xi _1' - \zeta '_Bv_1\right) ^2
\lesssim  \int _{|x|\le A}\sech \left(  \frac{\kappa}{2} x\right)   (  \xi _1 ^{\prime 2} + \xi _1 ^{  2}    )    .
\end{align*}
We have, thanks to \eqref{eq:KM1},
\begin{multline*}
 \int _{|x|\ge A}\sech \left(  \kappa x\right) \left(  v_1 ^{\prime 2}+     v_1^2\right) \le  \sech \left(  \frac{\kappa}{2} A\right)\int _{\R }\sech \left( \frac{8}{A} x\right) \left(  v_1 ^{\prime 2}+     v_1^2\right) dx \\
%----------------------------------------
   \lesssim \sech \left(  \frac{\kappa}{2} A\right) \varepsilon ^{-\pare{N+1}}\int _{\R }\sech \left( \frac{4}{A} x\right) \left(  \eta_1 ^{\prime 2}+     \eta_1^2\right) dx\le A ^{-1}\|\boldsymbol{\eta}\|_{\boldsymbol{ \Sigma }_A}^2.
\end{multline*}
Finally,  we claim the following, which  completes the proof of \eqref {eq:KMlemma3},
\begin{align} \label{eq:coercVB} & \int _{\R } \sech \left(  \frac{\kappa}{2} x\right)   (  \xi _1 ^{\prime 2} + \xi _1 ^{  2}    )  \lesssim  \int _{\R }( \xi _1 ^{\prime  2}+V_B \xi _1^{  2}   )\,dx .
\end{align}
In the case $5/3<p<2$, the above inequality  is true for all $\xi _1\in H ^{1}    \left( \R \right)$ and
follows easily from the fact  for $B$ large the potential $V_B$ is a small perturbation of the  positive potential  $- \  \frac{\varphi _B}{\zeta _B^2}V'_{N+1}
  >0$ for $x\neq 0$. In the case $p=2$, then  $V _{N+1}=V_3=0$, and so $V_B$ is a small potential.  Then for $B$ large, the coercivity in \eqref{eq:coercVB} is true
 only for $\xi _1\in H ^{1} _{\odd}  \left( \R \right)$.    Fortunately, from \eqref{def:xi} we see that $\xi _1$ has the same parity of $v_1$ and, as we remarked right under  \eqref{def:vBg},   $v_1$ is odd if  $p=2$.

\noindent We next have the following, which is \cite[Lemma 4]{KM22}  (see also the argument  in  \cite{CMS2023}, here $A^{-1} B\varepsilon^{-N-1}\ll A^{-\frac{1}{2}}$),
\begin{align}\label{eq:KMlemma4}
|D_{21}|\lesssim A^{-1/2}\left(\|\boldsymbol{\eta}\|    _{\boldsymbol{ \Sigma }_A}^2+\|\sech \left( \kappa x\right)\eta_1\|_{L^2}^2\right)   .
\end{align}
  By Lemma \ref{lem:modbound}, $ \av{\tilde{S}_{A,B}^\ast \cT D_\bmz\bmPhi\pare{\bmz}} \sim e^{-\kappa\av{x}} $ and \eqref{eq:KM1} we have that
  \begin{equation}
  \label{eq:D1}
  \begin{aligned}
|D_1|&\lesssim |\dot{{\bm z}}-\tilde{{\bm z}}| \|  \sech \left( 2\kappa x\right)    {\bm v}\|_{L^2}\\
& \lesssim \delta ^{p-1} \|\sech \left(  \kappa x\right) {\eta}_1\|_{L^2} \|\sech \left( 2\kappa x\right) {\bm v}\|_{L^2}   \lesssim \delta ^{p-1} \varepsilon ^{-\pare{N+1}} \|\sech \left(  \kappa x\right) \boldsymbol{\eta} \|_{L^2}^2 .
\end{aligned}
  \end{equation}
Like in \cite{CMS2023},  we have (obviously here  we care only of case $2>p>5/3$ where $V_{N+1}\neq 0$)
\begin{equation}
\label{eq:D3}
\begin{aligned}
|D_3|  & \ =\av{\< \comm{\<\im \varepsilon \partial_x\>^{-\pare{N+1}}}{V_4} \mathcal{A}^*\eta _1, \widetilde{S} _{A,B}v_1\>}  \\&
%------------------------------------------------
\le \norm{ \cosh \left( \frac{\kappa}{2} x\right)\comm{\<\im \varepsilon \partial_x\>^{-\pare{N+1}}}{V_4} \mathcal{A}^*\eta _1 } _{L^2} \norm{ \sech \left( \frac{\kappa}{2} x\right) \widetilde{S} _{A,B}v_1} _{L^2}
\\&
\leq   \varepsilon   \norm{ \sech \left( \frac{\kappa}{2} x\right)v_1}_{L^2}  \left(\norm{\sech \left( \frac{\kappa}{2} x\right) v_1'}_{L^2}+\norm{\sech \left( \frac{\kappa}{2} x\right) v_1}_{L^2}\right)\\&
\lesssim\varepsilon \left(\norm{ \sech \left( \frac{\kappa}{2} x\right)v_1'}_{L^2}^2+\norm{\sech \left( \frac{\kappa}{2} x\right) v_1}_{L^2}^2\right) ,
\end{aligned}
\end{equation}
where the upper bound can be absorbed inside the left hand side of \eqref{eq:lem:2v1}.

 \noindent   We now write $D_4= D _{41}+ D _{42}$, with \footnote{Notice that in the estimate of the term $D_4$ in \cite[Lemma 8.1]{CMS2023} the operator $\mathcal{T}$ is omitted by mistake. The proof can be corrected proceeding like here.}
 \begin{align*}&
D _{41} \defeq\<\mathcal{T}   \left(  f'(\boldsymbol{\Phi}\pare{\bm z}_1) - f'(Q) \right) \eta _1
 ,\tilde{S}_{A,B} {v}_1\>    \\& D _{42} \defeq\<\mathcal{T}   \left(  \left( f\left( \boldsymbol{\Phi}\pare{\bm z}_1 +\eta _1 \right)-f(\boldsymbol{\Phi}\pare{\bm z}_1)- f'(\boldsymbol{\Phi}\pare{\bm z}_1)\eta _1  \right)   \right)
 ,\tilde{S}_{A,B} {v}_1\>  .
\end{align*}
 By  $\sech \left( \frac{4}{A}  x \right)\sim 1$ on               $\supp \psi_{A,B}\subseteq [-2A,2A] $,  by \cref{eq:estder1,eq:estder2}
 and by \cref{eq:KM1,eq:KM2},
 \begin{equation}
 \label{eq:D41}
 \begin{aligned}
   |D _{41}|& \le \norm{    \sech \left( \frac{4}{A}  x \right) \mathcal{T} \left(  f'(\boldsymbol{\Phi}\pare{\bm z}_1) - f'(Q) \right) \eta _1 } _{L^2}
    \left(  \norm{\sech \left(  \frac{4}{A} x \right) v_1}_{L^2}
+ \norm{\sech \left(  \frac{4}{A} x \right)v_1'}_{L^2}\right)  \\&  \lesssim \varepsilon ^{-2\pare{N+1}}  \norm{    \left(  f'(\boldsymbol{\Phi}\pare{\bm z}_1) - f'(Q) \right) \eta _1 } _{L^2} \left(  \norm{\sech \left(  \frac{2}{A} x \right) \eta_1}_{L^2}
+ \norm{\sech \left(  \frac{2}{A} x \right)\eta_1'}_{L^2}\right)
\\& \lesssim   \varepsilon ^{-2\pare{N+1}}A |{\bm z}|   \norm{    e^{-\frac{p-1}{2}|x|} \eta _1 }_{L^2} \norm{\boldsymbol{\eta}}_{\boldsymbol{ \Sigma }_A}
 \lesssim \varepsilon ^{-2\pare{N+1}}A ^2 \delta   \norm{\boldsymbol{\eta}}_{\boldsymbol{ \Sigma }_A}^2 .
\end{aligned}
\end{equation}
We have similarly, using \eqref{eq:KM1} to get the second line and  \cref{eq:buuuuh1,eq:buuuuh2} the third,
   \begin{align} \nonumber
   |D _{42}|& \le  \varepsilon ^{- \pare{N+1}}A \norm{   \sech \left( \frac{4}{A}  x \right)  \mathcal{T}   \left( f\left( \boldsymbol{\Phi}\pare{\bm z}_1 +\eta _1 \right)-f(\boldsymbol{\Phi}\pare{\bm z}_1)- f'(\boldsymbol{\Phi}\pare{\bm z}_1)\eta _1  \right)  }_{L^2}
     \norm{ \boldsymbol{\eta}}_{\boldsymbol{ \Sigma }_A}
     \\&  \nonumber \lesssim  \varepsilon ^{- 2\pare{N+1}}A \norm{   \sech \left( \frac{2}{A}  x \right) \left( f\left( \boldsymbol{\Phi}\pare{\bm z}_1 +\eta _1 \right)-f(\boldsymbol{\Phi}\pare{\bm z}_1)- f'(\boldsymbol{\Phi}\pare{\bm z}_1)\eta _1  \right)  }_{L^2} \norm{ \boldsymbol{\eta}}_{\boldsymbol{ \Sigma }_A}
     \\& \nonumber \lesssim   \varepsilon ^{- 2\pare{N+1}}A \norm{   \sech \left( \frac{2}{A}  x \right)  \av{\eta _1}^{p}  }_{L^2} \norm{ \boldsymbol{\eta}}_{\boldsymbol{ \Sigma }_A} \le \varepsilon ^{- 2\pare{N+1}}A^2
     \| \eta _1 \| _{L^\infty}^{p-1}\norm{ \boldsymbol{\eta}}_{\boldsymbol{ \Sigma }_A}^2 \\& \lesssim \varepsilon ^{- 2\pare{N+1}}A^2
    \delta ^{p-1}\norm{ \boldsymbol{\eta}}_{\boldsymbol{ \Sigma }_A}^2 . \label{eq:D42}
     \end{align}
Hence \cref{eq:D41,eq:D42} yield
\begin{equation}
\label{eq:D4}
\av{D_4} \lesssim \pare{ \frac{A}{\varepsilon^4} }^2    \delta^{p-1}
     \|\boldsymbol{\eta}\|    _{\boldsymbol{ \Sigma }_A}^2.
\end{equation}
Finally,  we consider
\begin{align*}
 D_5 =
 \< \mathcal{T} D_{\bm z}\boldsymbol{\Phi}\pare{\bm z}\tilde{{\bm z}}_R ,\tilde{S}_{A,B}{\bm v}\> -\< \mathcal{T} \left(  f(\boldsymbol{\Phi}\pare{\bm z}_1)   - f(Q)-f' (Q)  \tilde{\boldsymbol{\Phi}}[{\bm z}]_1  \right) ,\tilde{S}_{A,B} {v}_1\> =:D_{51}+D_{52}.
\end{align*}
 We have
\begin{align*}
 D_{51} = & \  D_{511}+ D_{512} \\
%-----------------------------------------
   = & \  - \< \partial _x \mathcal{T} D_{\bm z}\boldsymbol{\Phi}\pare{\bm z}\tilde{{\bm z}}_R , \psi_{A,B} {\bm v}\>   - \frac{1}{2} \< \mathcal{T}  D_{\bm z}\boldsymbol{\Phi}\pare{\bm z}\tilde{{\bm z}}_R , \psi_{A,B}' {\bm v}\>
%   - \<        \< \im \varepsilon \partial _x \> ^{-\pare{N+1}} [\partial _x, \mathcal{A}^*]    D_{\bm z}\boldsymbol{\Phi}\pare{\bm z}\tilde{{\bm z}}_R , \psi_{A,B} {\bm v}\>
    .
\end{align*}
 We have
\begin{align*}
 | D_{511}|     \le     \norm{ \cosh \left(  \kappa x \right) \partial _x \mathcal{T} D_{\bm z}\boldsymbol{\Phi}\pare{\bm z}\tilde{{\bm z}}_R } _{L^2}
    \| \sech \left(  \kappa x \right) {\bm v} \| _{L^2}   .
 \end{align*}
 In \cite[Lemma 5.6]{CM2022}   it is shown  that there exist constants $C_0>0$ and $ \varepsilon _0>0$ such that for   $\varepsilon  \in (0, \varepsilon _0)$   and for $K_{\varepsilon}(x,y) \in \mathcal{D}'(\R \times \R)$   the Schwartz kernel of  $\mathcal{T} $, then  we have
\begin{align}  \label{eq:scker}
| K_{\varepsilon}(x,y)    | \le C_0  e^{-\frac{|x-y|}{3\varepsilon}}  \text{  for all $x,y$ with $|x-y|\ge 1$.}
\end{align}
 Using this fact, it  is elementary to show that $ \av{\partial _x \mathcal{T} D_{\bm z}\boldsymbol{\Phi}\pare{\bm z}} \sim e^{-k_2\av{x}} $. Hence, if $ 0 <  \kappa \ll k_2 $.
 \begin{align*}
 \av{ D_{511} }\lesssim |\tilde{{\bm z}}_R|.
 \end{align*}
The same argument yields the same bounds for $D_{512}$ so
  we conclude that
 \begin{align}\label{eq:D51}
 \av{ D_{51}}      \lesssim    |\tilde{{\bm z}}_R| \
    \|  {\bm v} \| _{L^2_{-\kappa }}  \lesssim \delta _2^{-1}|{\bm z}|^4 + \delta _2 \|  {\bm v} \| _{L^2_{-\frac{\kappa}{2}}}^2.
 \end{align}
 { Turning to $D _{52}$, we have
 \begin{align*}&
 | D_{52}|     \le      \|  \mathbf{C}  \| _{L^2 (|x|\le 2A)}
    \| \sech \left( \kappa x \right)  \left(   2^{-1}\psi_{A,B}' v_1+\psi_{A,B}v_1'    \right)   \| _{L^2}   \text{  where } \\&\mathbf{C}
\defeq
  \cosh \left(  \kappa x \right) \mathcal{T} \left(  f(\boldsymbol{\Phi}\pare{\bm z}_1)   - f(Q)-f' (Q)  \tilde{\boldsymbol{\Phi}}[{\bm z}]_1  \right) .
 \end{align*}
 For a cutoff $\chi$ like in \Cref{notation:chi},  we write  $ \mathbf{C } =\mathbf{C }_{1}+\mathbf{C }_{2}$, with
 \begin{align*}
   & \mathbf{C }_{1}\defeq  \cosh \left( \kappa x \right) \mathcal{T}   \chi (x-\cdot ) \left(  f(\boldsymbol{\Phi}\pare{\bm z}_1)   - f(Q)-f' (Q)  \tilde{\boldsymbol{\Phi}}[{\bm z}]_1  \right) ,\\&  \mathbf{C }_{2}\defeq  \cosh \left( \kappa x \right) \mathcal{T}   \chi _1(x-\cdot ) \left(  f(\boldsymbol{\Phi}\pare{\bm z}_1)   - f(Q)-f' (Q)  \tilde{\boldsymbol{\Phi}}[{\bm z}]_1  \right) \text{ with }\chi _1:=1-\chi .
 \end{align*}
 For the $I$ in \eqref{eq:defI},  by    \eqref{eq:Ijest2} and using  \eqref{eq:scker},  we have
 \begin{align}\label{eq:estbfc2}
   \|  \mathbf{C} _2 \| _{L^2  } \lesssim \| \cosh \left(  \kappa x \right) \int   e^{-\frac{|x-y|}{3\varepsilon}} |I (y)| dy\| _{L^2  } \lesssim  |{\bm z}|^2\| \cosh \left(  \kappa x \right) \int   e^{-\frac{|x-y|}{3\varepsilon}} e^{-|y| } dy\| _{L^2  }  \lesssim  |{\bm z}|^2,
 \end{align}
 We bound
 \begin{align*}&
  \|  \mathbf{C} _2 \| _{L^2 (|x|\le 2A)  } \\&\lesssim  \int _{\R} d\xi   | \widehat{\chi}(\xi) |\sum _{j=0 }^{ \log (10A)}
     2^{ \kappa j} \|    \varphi _j    \mathcal{T}   e ^{\im \xi \sqcup }    \widetilde{\varphi} _j   \left(  f(\boldsymbol{\Phi}\pare{\bm z}_1)   - f(Q)-f' (Q)  \tilde{\boldsymbol{\Phi}}[{\bm z}]_1  \right) \| _{L^2 (|x|\le 2A)}
 \end{align*}
 for appropriate Paley--Littlewood decompositions  (in $x$ and $y$ space) with $ \supp  \varphi _j $ and $ \supp  \widetilde{\varphi} _j $ in $|x|\sim 2 ^{j}$ for $j\ge 1$ and   $|x|\lesssim 1$ for $j=0$.
  Since for $|y|\lesssim 2 ^ j$  with $j\le   \log (10A)$ we have
 \begin{align} \label{eq:estbfc2-}&
   f(\boldsymbol{\Phi}\pare{\bm z}_1)   - f(Q)-f' (Q)  \tilde{\boldsymbol{\Phi}}[{\bm z}]_1  = 2^{-1} f'' (Q)  \left( \tilde{\boldsymbol{\Phi}}[{\bm z}]_1\right) ^2 +  {E} \text{  with}\\& |E|\lesssim \left | f''' (Q)  \left( \tilde{\boldsymbol{\Phi}}[{\bm z}]_1\right) ^3 \right | \lesssim |{\bm z}|^3 e^{-\frac{3-p}{2} |x|  } ,\nonumber
 \end{align}
 we accordingly split  $\mathbf{C} _2= \mathbf{C} _{21}+\mathbf{C} _{22} $.  With the function   $ \( e ^{\im \xi \sqcup }\) (x):=  e ^{\im \xi x }$, we have
 \begin{align*}&
  \|  \mathbf{C} _{21} \| _{L^2 (|x|\le 2A)  } \lesssim    \int _{\R} d\xi   | \widehat{\chi}(\xi) |\sum _{j=0 }^{ \log (10A)}
     2^{ \kappa j} \|   \< \im \varepsilon \partial _x \> ^{-\pare{N+1}}    \mathcal{A }^*    e ^{\im \xi \sqcup }    \widetilde{\varphi} _j   f'' (Q)  \left( \tilde{\boldsymbol{\Phi}}[{\bm z}]_1\right) ^2 \| _{L^2 (|x|\le 2A)}\\& \lesssim  \|  \< \xi \>  ^{N+1}  \widehat{\chi}  \|_{L^1\left(\R _\xi \right)}
     \sum _{j=0 }^{ \log (10A)} 2^{ \kappa j} \|  \widetilde{\varphi} _j   f'' (Q)  \left( \tilde{\boldsymbol{\Phi}}[{\bm z}]_1\right) ^2 \| _{H^{N+1}(\R )}  \lesssim |{\bm z}|^2.
 \end{align*}
 We have
 \begin{align*}&
  \|  \mathbf{C} _{22} \| _{L^2 (|x|\le 2A)  } \lesssim   \int _{\R} d\xi   | \widehat{\chi}(\xi) |\sum _{j=0 }^{ \log (10A)}
     2^{\kappa  j} \|   \< \im \varepsilon \partial _x \> ^{-\pare{N+1}}    \mathcal{A }^*    e ^{\im \xi \sqcup }    \widetilde{\varphi} _j  E  \| _{L^2 (|x|\le 2A)}\\& \lesssim  \|     \widehat{\chi}  \|_{L^1 \left(\R _\xi \right)} \varepsilon^{-\pare{N+1}}
     \sum _{j=0 }^{ \log (10A)} 2^{  \kappa  j} \| E  \| _{L^2(\R )}  \lesssim  A \varepsilon^{-\pare{N+1}}  |{\bm z}| ^3 \lesssim  A \varepsilon^{-\pare{N+1}} \delta  |{\bm z}| ^2.
 \end{align*}
 Summing up, we have
\begin{align}\label{eq:D52}
   \av{ D_{52}}  \lesssim \delta _2^{-1}|{\bm z}|^4 + \delta _2 \|  {\bm v} \| _{L^2_{-\frac{\kappa}{2}}}^2 +   \delta _2 \norm{ \sech \left(  \frac{\kappa}{2} x\right)   {v}'_1 } _{L^2  } ^2
\end{align}}
\Cref{eq:D51,eq:D52} give that
\begin{equation}\label{eq:D5}
\av{ D_{5}}  \lesssim \frac{\av{\bm z}^4}{\delta _2} + \delta _2 \|  {\bm v} \| _{L^2_{-\frac{\kappa}{2}}}^2 +   \delta _2 \norm{ \sech \left(  \frac{\kappa}{2} x\right)   {v}'_1 } _{L^2  } ^2
\end{equation}
Collecting   \cref{eq:D1,eq:D3,eq:D4,eq:D5,eq:KMlemma4,eq:KMlemma3,def:xi} and bootstrapping, we deduce   \eqref{eq:lem:2v1}.
\end{proof}

\begin{lemma}\label{lem:2v2}
	We have
	\begin{align}\label{eq:lem:2v2}
		\|e^{-\kappa\<x\>/2}v_2\|_{L^2} + \dot{\mathcal{I}}_{\mathrm{2nd},2}\lesssim& \|e^{-\kappa \<x\>/2}v_1'\|_{L^2}^2+
\|e^{-\kappa \<x\>/2}v_1\|_{L^2}^2+|{\bm z}|^4+
   o_{\varepsilon}(1)
   \| \boldsymbol{\eta} \|_{ \boldsymbol{ \Sigma }_A}^2      .%+\varepsilon^{-1}A^2\delta\|e^{-\kappa\<x\>}\eta_1\|_{L^2}
	\end{align}
\end{lemma}

\begin{proof}
Differentiating $\mathcal{I}_{\mathrm{2nd},2}$, we have
\begin{align*}&
\dot{\mathcal{I}}_{\mathrm{2nd},2}=    -\< \mathcal{T}D\boldsymbol{\Phi}({\bm z})(\dot{{\bm z}}-\tilde{{\bm z}}),\sigma_3 e^{-\kappa \<x\>}{\bm v}\> \\& -
\<{\bf L}_{ N+1}{\bm v},\tilde{S}_{A,B}{\bm v}\> +\<\begin{pmatrix}
\comm{\< \ii\varepsilon \partial_x\>^{-\pare{N+1}}}{V_4} & 0 \\ 0 & 0
\end{pmatrix}\mathcal{A}^*\boldsymbol{\eta},\sigma_3 e^{-\kappa \<x\>}{\bm v}\>
  \\&+\<\mathcal{T} \left( \left(  f'(\boldsymbol{\Phi}\pare{\bm z}_1) - f'(Q) \right) \eta _1   + \left( f\left( \boldsymbol{\Phi}\pare{\bm z}_1 +\eta _1 \right)-f(\boldsymbol{\Phi}\pare{\bm z}_1)- f'(\boldsymbol{\Phi}\pare{\bm z}_1)\eta _1  \right)   \right)
 ,  e^{-\kappa \<x\>}  {v}_1\>  \\& -\<\mathcal{T}  {\bm R}\pare{\bm z} ,\sigma_3 e^{-\kappa \<x\>}{\bm v}\>
=:E_1+E_2+E_3+E_4+E_5.
\end{align*}
Like in \cite{CMS2023} we have
\begin{align*}&
	E_2=-\|e^{-\kappa\<x\>/2}v_2\|_{L^2}^2+\<L_{N+1}v_1, e^{-\kappa\<x\>}v_1\>=-\|e^{-\kappa\<x\>/2}v_2\|_{L^2}^2+E_{21}, \text{  with}\\&
	|E_{21}|\lesssim  \|e^{-\kappa \<x\>/2}v_1'\|_{L^2}^2+\|e^{-\kappa \<x\>/2}v_1\|_{L^2}^2.
\end{align*}
By Lemma \ref{lem:modbound}, we have
\begin{align*}
	|E_1|\lesssim \delta ^{p-1}\|e^{-\kappa\<x\>/2}{\bm v}\|_{L^2}  \|e^{-\kappa\<x\>}\boldsymbol{\eta}\|_{L^2}
 \lesssim \delta^{p-1} \varepsilon ^{-N}\|e^{-\kappa\<x\>/2}{\bm v}\|_{L^2}^2.
\end{align*}
By \eqref{eq:KM3b}, we have
\begin{align*}
	&|E_3|  = \av{\<  [\<\im\varepsilon\partial_x\>^{-N},V_4]\mathcal{A}^*\eta_1,\sigma_3e^{-\kappa\<x\>}v_1\> } \lesssim
 \varepsilon \|e^{-\frac{\kappa}{2}\<x\> }v_1\|_{L^2} \|e^{-\kappa\<x\> }v_1\|_{L^2}    \le  \varepsilon \|e^{-\frac{\kappa}{2}\<x\> }v_1\|_{L^2}^2.
\end{align*}
Proceeding like in Lemma \ref{lem:2v1}, we write
$E_4= E _{41}+ E _{42}$, with \footnote{Also in the estimate of the term $E_4$ in \cite[Lemma 8.2]{CMS2023}, like in the estimate of
$D_4$ in \cite[Lemma 8.1]{CMS2023},
the operator $\mathcal{T}$ has been omitted by mistake. The proof can be corrected proceeding like here.}
 \begin{align*}&
E _{41} \defeq\<\mathcal{T}   \left(  f'(\boldsymbol{\Phi}\pare{\bm z}_1) - f'(Q) \right) \eta _1
 ,e^{-\kappa \<x\>} {v}_1\>    \\& E _{42} \defeq\<\mathcal{T}   \left(  \left( f\left( \boldsymbol{\Phi}\pare{\bm z}_1 +\eta _1 \right)-f(\boldsymbol{\Phi}\pare{\bm z}_1)- f'(\boldsymbol{\Phi}\pare{\bm z}_1)\eta _1  \right)   \right)
 ,e^{-\kappa \<x\>} {v}_1\>  .
\end{align*}
 By  an  analogue of \eqref{eq:key1}
 and by \eqref{eq:KM1}--\eqref{eq:KM2},
 \begin{align*}
   |E _{41}|& \le \varepsilon ^{- N-1}  \|      \left(  f'(\boldsymbol{\Phi}\pare{\bm z}_1) - f'(Q) \right) \eta _1 \| _{L^2}
       \|e^{-\kappa \<x\>} {v}_1\|_{L^2}   \\&   \lesssim   \varepsilon ^{-\pare{N+1}}A |{\bm z}|   \|    e^{-\frac{p-1}{2}|x|} \eta _1 \| _{L^2} \|e^{-\kappa \<x\>} {v}_1\|_{L^2}
 \lesssim   \varepsilon ^{-N}A \delta   \|    e^{-\frac{p-1}{4}|x|} v _1 \| _{L^2} \|e^{-\kappa \<x\>} {v}_1\|_{L^2}  \\&   \lesssim     \varepsilon ^{- N-1}A  \delta  \|e^{-\kappa \<x\>} {v}_1\|_{L^2} ^2 .
\end{align*}
In analogy to estimates for  $D_{41} $ and $D_{42} $ in Lemma \ref{lem:2v1}, but using $\sech \left( \frac{4}{A}  x \right)e^{-\frac{\kappa }{2} \<x\>}\lesssim e^{-\kappa \<x\>}$ along with
  \eqref{eq:KM1}--\eqref{eq:KM2},
 \begin{align*}
   |E _{42}|& \le \|   \sech \left( \frac{4}{A}  x \right) \mathcal{T} \left(   f\left( \boldsymbol{\Phi}\pare{\bm z}_1 +\eta _1 \right)-f(\boldsymbol{\Phi}\pare{\bm z}_1)- f'(\boldsymbol{\Phi}\pare{\bm z}_1)\eta _1 \right)  \| _{L^2}
      \|  e^{-\frac{\kappa }{2} \<x\>} v_1\|_{L^2} \\&
        \lesssim   \varepsilon ^{-2\pare{N+1}}A  \|   \sech \left( \frac{2}{A}  x \right) \left( f\left( \boldsymbol{\Phi}\pare{\bm z}_1 +\eta _1 \right)-f(\boldsymbol{\Phi}\pare{\bm z}_1)- f'(\boldsymbol{\Phi}\pare{\bm z}_1)\eta _1  \right)  \| _{L^2} \|e^{-\kappa \<x\>} {v}_1\|_{L^2}
 \\& \lesssim   \varepsilon ^{-2\pare{N+1}}A \delta ^{p-1}  \|\boldsymbol{\eta}\|    _{\boldsymbol{ \Sigma }_A} \|  e^{-\frac{\kappa }{2} \<x\>} v_1\|_{L^2}    \le      \varepsilon ^{-2\pare{N+1}}A \delta ^{p-1} \left(  \|\boldsymbol{\eta}\|    _{\boldsymbol{ \Sigma }_A} ^2+  \|  e^{-\frac{\kappa }{2} \<x\>} v_1\|_{L^2} ^2    \right).
\end{align*}
 Finally,  we consider
\begin{align*}
 E _5 =
 \< \mathcal{T} D_{\bm z}\boldsymbol{\Phi}\pare{\bm z}\tilde{{\bm z}}_R ,\sigma_3 e^{-\kappa \<x\>}{\bm v}\> -\< \mathcal{T} \left(  f(\boldsymbol{\Phi}\pare{\bm z}_1)   - f(Q)-f' (Q)  \tilde{\boldsymbol{\Phi}}[{\bm z}]_1  \right) ,  e^{-\kappa \<x\>}  {v}_1\> =:E_{51}+E_{52}.
\end{align*}
The last three terms can be bounded similarly. We have
\begin{align*}
 | E_{51}|     \le     \|  \mathcal{T} D_{\bm z}\boldsymbol{\Phi}\pare{\bm z}\tilde{{\bm z}}_R \| _{L^2}
    \|e^{-\kappa \<x\>} {\bm v} \| _{L^2}  \lesssim  \|   \mathcal{A}^* D_{\bm z}\boldsymbol{\Phi}\pare{\bm z}\tilde{{\bm z}}_R \| _{L^2}
    \|e^{-\kappa \<x\>} {\bm v} \| _{L^2} \lesssim    \delta _2^{-1}|{\bm z}|^4 + \delta _2 \| e^{-\frac{\kappa}{2} \<x\>} {\bm v} \| _{L^2 }^2.
 \end{align*}
 Turning to $E _{52}$, we have
 \begin{align*}&
 | E_{52}|     \le      \|  \mathbf{E}  \| _{L^2  }
    \| e^{-\frac{\kappa}{2} \<x\>}v_1         \| _{L^2}   \text{  where } \\&\mathbf{E}
\defeq
  e^{-\frac{\kappa}{2} \<x\>} \mathcal{T} \left(  f(\boldsymbol{\Phi}\pare{\bm z}_1)   - f(Q)-f' (Q)  \tilde{\boldsymbol{\Phi}}[{\bm z}]_1  \right) .
 \end{align*}
 Let us write  $ \mathbf{E } =\mathbf{E }_{1}+\mathbf{E}_{2}$ with
 \begin{align*}
   & \mathbf{E}_{1}\defeq  e^{-\frac{\kappa}{2} \<x\>}\mathcal{T}   \chi (x-\cdot ) \left(  f(\boldsymbol{\Phi}\pare{\bm z}_1)   - f(Q)-f' (Q)  \tilde{\boldsymbol{\Phi}}[{\bm z}]_1  \right) ,\\&  \mathbf{E }_{2}\defeq  e^{-\frac{\kappa}{2} \<x\>} \mathcal{T}   \chi _1(x-\cdot ) \left(  f(\boldsymbol{\Phi}\pare{\bm z}_1)   - f(Q)-f' (Q)  \tilde{\boldsymbol{\Phi}}[{\bm z}]_1  \right) \text{ with }\chi _1=1-\chi .
 \end{align*}
 Like for \eqref{eq:estbfc2} and using \eqref{eq:Ijest2},   we have
 \begin{align*}
   \|  \mathbf{E} _2 \| _{L^2  } \lesssim \|  e^{-\frac{\kappa}{2} \<x\>} \int   e^{-\frac{|x-y|}{3\varepsilon}} |I (y)|dy\| _{L^2  } \lesssim  |{\bm z}|^2\|  e^{-\frac{\kappa}{2} \<x\>} \int   e^{-\frac{|x-y|}{3\varepsilon}} e^{-|y| } dy\| _{L^2  }  \lesssim  |{\bm z}|^2.
 \end{align*}
 In analogy to the bound of $\mathbf{C} _2$   in Lemma \ref{lem:2v1}, we bound
 \begin{align*}&
  \|  \mathbf{E} _2 \| _{L^2 } \lesssim  \int _{\R} d\xi   | \widehat{\chi}(\xi) |\sum _{j=0 }^{ \infty}
     2^{ -\frac{\kappa}{2} j} \|    \varphi _j    \mathcal{T}   e ^{\im \xi \sqcup }    \widetilde{\varphi} _j   \left(  f(\boldsymbol{\Phi}\pare{\bm z}_1)   - f(Q)-f' (Q)  \tilde{\boldsymbol{\Phi}}[{\bm z}]_1  \right) \| _{L^2  }
 \end{align*}
 for appropriate Paley--Littlewood decompositions with $ \supp  \varphi _j $ and $ \supp  \widetilde{\varphi} _j $ in $|x|\sim 2 ^{j}$ for $j\ge 1$ and   $|x|\lesssim 1$ for $j=0$.
 Next we split $  \|  \mathbf{E} _2 \| _{L^2  } \le A _{2l}+ A _{2h}$, with $A _{2l}$ involving the sum   for
   $j\le  \log (10A)$   and  $A _{2l}$ involving the sum   for
   $j>    \log (10A)$.   Now, using the same argument in  Lemma \ref{lem:2v1}, we have
    \begin{align*}
     A _{2l} \lesssim |{\bm z}|^2.
   \end{align*}
  Next, using also \eqref{eq:Ijest} and  \eqref{eq:Ijest2}
  \begin{align*}
     A _{2h}&  \lesssim  \int _{\R} d\xi   | \widehat{\chi}(\xi) |\sum _{j> \log (10A) }
     2^{ -\frac{\kappa}{2} j} \|         \mathcal{T}   e ^{\im \xi \sqcup }    \widetilde{\varphi} _j   \left(  f(\boldsymbol{\Phi}\pare{\bm z}_1)   - f(Q)-f' (Q)  \tilde{\boldsymbol{\Phi}}[{\bm z}]_1  \right) \| _{L^2  } \\& \lesssim  \varepsilon ^{-\pare{N+1}} A ^{-1}   \|        f(\boldsymbol{\Phi}\pare{\bm z}_1)   - f(Q)-f' (Q)  \tilde{\boldsymbol{\Phi}}[{\bm z}]_1   \| _{L^2  } \lesssim \varepsilon ^{-\pare{N+1}} A ^{-1} |{\bm z}|^2.
   \end{align*}
   Summing up
   \begin{align*}&  | E_{52}|     \le      \|  \mathbf{E}  \| _{L^2  }
    \| e^{-\frac{\kappa}{2} \<x\>}v_1         \| _{L^2} \lesssim  \delta _2^{-1}|{\bm z}|^4 + \delta _2 \| e^{-\frac{\kappa}{2} \<x\>}v_1 \| _{L^2 }^2
 \end{align*}
Collecting the estimates and bootstrapping,  we obtain the conclusion \eqref{eq:lem:2v2}.
\end{proof}

Combining Lemmas \ref{lem:2v1} and \ref{lem:2v2}, we have
\begin{lemma}\label{lem:2v3}
	For any $\mu>0$, we have
	\begin{align*}
		&\int_0^T\left(\| \sech \left( \frac{\kappa}{2}x \right)v_1'\|_{L^2}^2+\|\sech \left( \frac{\kappa}{2}x \right){\bm v}\|_{L^2}^2\right)
\lesssim  \delta   + o_\varepsilon (1)\int_0^T\| \boldsymbol{\eta} \|_{ \boldsymbol{ \Sigma }_A}^2+ \|{\bm z} \|_{L^4(0,T)}^4.
	\end{align*}
\end{lemma}
\begin{proof}
	The claim follows from Lemmas  \ref{lem:2v1} and \ref{lem:2v2} and
	\begin{align*}
		|\mathcal{I}_{\mathrm{2nd},1}|\lesssim B\varepsilon^{-\pare{N+1}}\delta^2,\
		|\mathcal{I}_{\mathrm{2nd},2}|\lesssim \varepsilon^{-\pare{N+1}}\delta^2.
	\end{align*}
\end{proof}

\begin{proof}[Proof of Proposition \ref{prop:2ndvirial}]
	It  is a consequence of Lemma \ref{lem:2v3}     and inequality \eqref{eq:key1}.
\end{proof}

\section*{Acknowledgments}
S.C. was supported   by the Prin 2020 project \textit{Hamiltonian and Dispersive PDEs} N. 2020XB3EFL.
M.M. was supported by the JSPS KAKENHI Grant Number 19K03579 and G19KK0066A.
S.S. was partially supported by Gnampa, Indam and the PRIN 2022 project {\it Turbulent Effects vs Stability in Equations from Oceanography}, acronym TESEO.

	\begin{footnotesize}
	%	\bibliography{references}
	%	\bibliographystyle{plain}
	
	\end{footnotesize}

Department of Mathematics, Informatics and Geosciences,  University
of Trieste, via Valerio  12/1  Trieste, 34127  Italy.
{\it E-mail Address}: {\tt scuccagna@units.it}

Department of Mathematics and Informatics,
Graduate School of Science,
Chiba University,
Chiba 263-8522, Japan.
{\it E-mail Address}: {\tt maeda@math.s.chiba-u.ac.jp}

Department of Mathematics, Informatics and Geosciences,  University
of Trieste, via Valerio  12/1  Trieste, 34127  Italy
and
Dipartimento di Matematica Federigo Enriques, Università Statale di Milano, Via Saldini 50, 20133, Milano
.
{\it E-mail Address}: {\tt federico.murgante@unimi.it}

Department of Mathematics, Informatics and Geosciences,  University
of Trieste, via Valerio  12/1  Trieste, 34127  Italy.
{\it E-mail Address}: {\tt STEFANO.SCROBOGNA@units.it}

\end{document}